\documentclass[12pt,reqno]{amsart}
\usepackage{amsmath}
\usepackage{mathtools}

\DeclareMathOperator*{\argmax}{arg\,max}
\usepackage{graphicx} 
\usepackage{microtype}
\usepackage[british]{babel}
\usepackage[utf8]{inputenc}
\usepackage{imakeidx}
\usepackage{subfigure}
\usepackage{graphicx}
\usepackage{array}
\usepackage[dvipsnames,table]{xcolor}
\usepackage[comma,sort&compress]{natbib}

\usepackage{xargs}                      % Use more than one optional parameter in a new commands
\newcommandx{\unsure}[2][1=]{\todo[linecolor=red,backgroundcolor=red!25,bordercolor=red,#1]{#2}}
\newcommandx{\change}[2][1=]{\todo[linecolor=blue,backgroundcolor=blue!25,bordercolor=blue,#1]{#2}}
\newcommandx{\info}[2][1=]{\todo[linecolor=OliveGreen,backgroundcolor=OliveGreen!25,bordercolor=OliveGreen,#1]{#2}}
\newcommandx{\improvement}[2][1=]{\todo[linecolor=Plum,backgroundcolor=Plum!25,bordercolor=Plum,#1]{#2}}
\newcommandx{\thiswillnotshow}[2][1=]{\todo[disable,#1]{#2}}
\usepackage[hyperindex,breaklinks]{hyperref}
\hypersetup{ colorlinks=true, linkcolor=blue, citecolor=blue,
  filecolor=blue, urlcolor=blue}\usepackage{blindtext}
\usepackage{MnSymbol}
\usepackage{appendix}
\usepackage{tikz}
\usetikzlibrary{arrows.meta,calc,positioning,patterns,decorations.pathmorphing}
\definecolor{graphgreen}{RGB}{30,110,74}
\definecolor{graphred}{RGB}{170,38,44}
\definecolor{graphblue}{RGB}{35,55,142}
\definecolor{graphpurple}{RGB}{77,39,108}

\usepackage{mathtools}
\usepackage{dsfont}
\usepackage{bbm}
\usepackage{enumitem}
\usepackage[final]{showkeys}
\usepackage[margin=2.5cm]{geometry}

\numberwithin{equation}{section}
 \makeatletter
\def\paragraph{\@startsection{paragraph}{4}%
  \z@\z@{-\fontdimen2\font}%
  {\bfseries\itshape}}
  \makeatother

\newtheorem{thm}{Theorem}[section]
\newtheorem*{thm*}{Theorem}
\newtheorem{corollary}{Corollary}[section]
\newtheorem{prop}{Proposition}[section]
\newtheorem{lemma}{Lemma}[section]

\theoremstyle{definition}

\newtheorem{remark}{Remark}[section]

\DeclareMathOperator{\argmin}{\mathrm{argmin}}

\newcommand{\Var}{\mathrm{Var}}

\newcommand{\floor}[1]{\left\lfloor #1\right\rfloor}
\newcommand{\abs}[1]{\left| #1\right|}

\newcommand{\ONE}{\mathbbm{1}}
\newcommand{\bbP}{\mathbb{P}}%UST probability
\newcommand{\PRW}{\mathcal{P}}%RW probability
\newcommand{\bbE}{\mathbb{E}}
\newcommand{\clE}{\mathcal{E}}

\newif\ifrevisions
\revisionstrue
\makeatother

\begin{document}

\title{Avalanches of the Abelian sandpile on the hierarchical lattice}

\author{Luca Cattani}
\address{%
University College London, Department of Statistical Science, London, United Kingdom.
}
\email{luca.cattani.25@ucl.ac.uk}
\author{Leandro Chiarini}
\address{%
  Instituto de Matemática e Estatística,  Universidade de São Paulo, São Paulo, Brazil.
}
\email{lchiarini@usp.br}
\author{Alessandra Cipriani}
\address{%
 University College London, Department of Statistical Science, London, United Kingdom.
}
\email{a.cipriani@ucl.ac.uk}
\date{\today}

\thanks{LC is supported by the
São Paulo Research Foundation (FAPESP) -- Process Number \#2026/10177-2. LC carried out part of this project while attending the Trimester Program “Probabilistic methods in quantum field theory” at the Hausdorff Research Institute for Mathematics, funded by the Deutsche Forschungsgemeinschaft (DFG, German Research Foundation) under Germany's Excellence Strategy --EXC-2047/1-- 390685813.
AC is supported by EPSRC New Investigator Grant UKRI3308. The authors would like to thank D. Dhar, J.--P. Eckmann, V. Jacquier, X. Li and W. Ruszel for helpful conversations.}

\begin{abstract}
We study the Abelian sandpile model (ASM) on a sequence of finite fractal graphs: the hierarchical, or Migdal--Kadanoff, lattices. {Our first main result consists of matching power-law tail bounds for the avalanche and the uniform wave started at certain} ``pivotal'' vertices. {We also establish a limiting result for the pivotal reach of the avalanche and, finally, a scaling limit for the size of a uniform wave started at a vertex of high degree.}
\end{abstract}
\maketitle

\section{Introduction}
The Abelian sandpile model is a fundamental example of self-{organised} criticality~\citep{PCK}: the stochastic system reaches criticality without an externally tuned temperature or density. The model is defined as follows. {On a finite connected graph with no self-loops and with a distinguished sink}, vertices hold grains; whenever a vertex contains {a number of grains greater than its graph degree}, it topples, sending one grain along each incident edge, while grains reaching the sink disappear. A new grain is added only after {stabilisation}. {\em Abelian} means that the final configuration and toppling counts do not depend on the order of topplings. Repeated addition and {stabilisation} produce the uniform stationary measure on recurrent configurations, and the study of criticality for this process concerns the large-system behaviour of the stationary process. 

{In this paper, the main observable of interest is the so-called} {\em avalanche cluster} $Av^x$: the set of vertices that topple after a grain is added at a site $x$, called the {\em source}. We will refer to the cluster simply as ``avalanche'' in what follows. Its size is the number of distinct toppled vertices. We are interested in
estimates of the form
\[
P(|Av^x|\geq s)\approx {s^{-\alpha}.}\]
 Thus $\alpha$ is the exponent of the cumulative tail. In the stationary ASM, avalanche sizes are expected to exhibit power-law tails in the thermodynamic limit, reflecting the scale invariance associated with self-{organised} criticality. The central question is how this
exponent depends on the geometry of the underlying graph and on the location of the source relative
to the sink~\citep{PCK,dhar2006theoretical}.\subsection*{Previous work and the setting of this paper}  The mean-field tail exponent $\alpha=1/2$ is established on regular trees~\citep{dhar1990Bethe} and complete graphs~\citep{janowsky1993exact} in their appropriate limiting {regimes. \citeauthor{jarai2020mean}} proved matching bounds on surviving supercritical {Bienaym\'e--Galton--Watson} trees when the offspring distribution has an exponential moment. Another work exhibiting avalanche-mass exponents on self-similar graphs is~\cite{matter2014abelian}. Earlier
numerical and scaling work on fractal graphs includes the study of the Sierpiński gasket by
\citeauthor{DaerdenVanderzande1998}. As for the Euclidean lattice, after earlier work by~\citeauthor{bhupatiraju2017inequalities}, \citeauthor{hutchcroft2020universality} showed for the avalanche cluster a tail exponent of $1/2$ on $\mathbb Z^d$ for $d\ge5$, with the same exponent for the avalanche volume, and extended mean-field results to broader graph classes, including transitive graphs of polynomial growth dimension exceeding four and bounded-degree nonamenable graphs, subject to the precise hypotheses and qualifications of those theorems. In dimension four, rigorous estimates identify the leading mean-field exponent with logarithmic corrections for several uniform spanning tree observables~\citep{HutchcroftSousi2023}. A recent preprint by~\citeauthor{li2026tailexponentsthreedimensionaluniform} establishes $ P(|Av|\ge s)\leq s^{-1/3+o(1)}$ on $\mathbb Z^3$\footnote{Personal communication with the authors.}. {The almost sure finiteness and a radius upper bound on two-dimensional avalanches have recently} been announced in~\cite{bourabee2026sandpileavalanchesdimensionsfinite}. 

We consider the ASM on diamond hierarchical graphs, also called {Migdal--Kadanoff} lattices~\citep{griffiths1982spin,kadanoff1976notes,berker1979renormalisation,kaufman1981exactly}. Starting
from a single edge with endpoints $a$ and $z$, each construction step replaces every edge by $b$
parallel paths, each containing $s$ edges, where $b,\,s\geq 2$. An example of a lattice $D_n^{b,s}$ is shown in Figure~\ref{fig:MKlattice}.
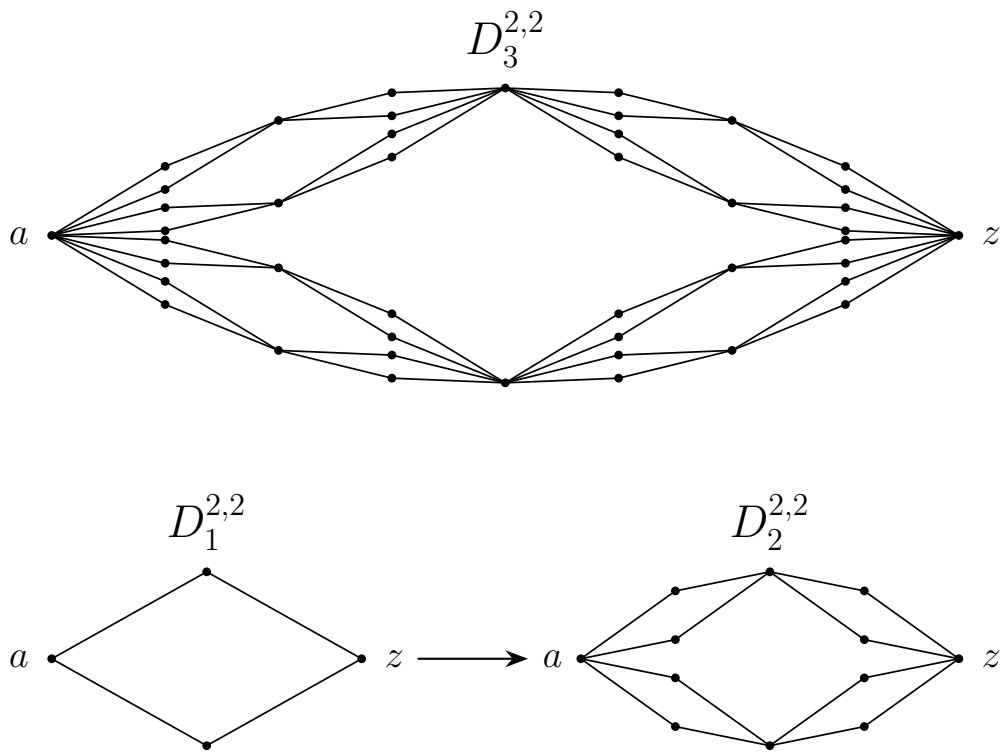
\begin{figure}[h!]
  \begin{tikzpicture}[x=1cm,y=1cm,line cap=round,line join=round,
  bond/.style={draw=black,line width=0.65pt},
  site/.style={circle,fill=black,inner sep=0pt,minimum size=3.3pt}]
% D_0 is one edge. Every step replaces each edge by a diamond.
% All vertices, including degree-two subdivision sites, are drawn.
% D_3: 44 vertices, 64 edges
\coordinate (top-0) at (0.000000,5.600000);
\coordinate (top-1) at (12.000000,5.600000);
\coordinate (top-2) at (6.000000,7.550000);
\coordinate (top-3) at (3.000000,7.121000);
\coordinate (top-4) at (1.500000,6.513380);
\coordinate (top-5) at (1.500000,6.207620);
\coordinate (top-6) at (4.500000,7.488380);
\coordinate (top-7) at (4.500000,7.182620);
\coordinate (top-8) at (3.000000,6.029000);
\coordinate (top-9) at (1.500000,5.967380);
\coordinate (top-10) at (1.500000,5.661620);
\coordinate (top-11) at (4.500000,6.942380);
\coordinate (top-12) at (4.500000,6.636620);
\coordinate (top-13) at (9.000000,7.121000);
\coordinate (top-14) at (7.500000,7.488380);
\coordinate (top-15) at (7.500000,7.182620);
\coordinate (top-16) at (10.500000,6.513380);
\coordinate (top-17) at (10.500000,6.207620);
\coordinate (top-18) at (9.000000,6.029000);
\coordinate (top-19) at (7.500000,6.942380);
\coordinate (top-20) at (7.500000,6.636620);
\coordinate (top-21) at (10.500000,5.967380);
\coordinate (top-22) at (10.500000,5.661620);
\coordinate (top-23) at (6.000000,3.650000);
\coordinate (top-24) at (3.000000,5.171000);
\coordinate (top-25) at (1.500000,5.538380);
\coordinate (top-26) at (1.500000,5.232620);
\coordinate (top-27) at (4.500000,4.563380);
\coordinate (top-28) at (4.500000,4.257620);
\coordinate (top-29) at (3.000000,4.079000);
\coordinate (top-30) at (1.500000,4.992380);
\coordinate (top-31) at (1.500000,4.686620);
\coordinate (top-32) at (4.500000,4.017380);
\coordinate (top-33) at (4.500000,3.711620);
\coordinate (top-34) at (9.000000,5.171000);
\coordinate (top-35) at (7.500000,4.563380);
\coordinate (top-36) at (7.500000,4.257620);
\coordinate (top-37) at (10.500000,5.538380);
\coordinate (top-38) at (10.500000,5.232620);
\coordinate (top-39) at (9.000000,4.079000);
\coordinate (top-40) at (7.500000,4.017380);
\coordinate (top-41) at (7.500000,3.711620);
\coordinate (top-42) at (10.500000,4.992380);
\coordinate (top-43) at (10.500000,4.686620);
\draw[bond] (top-0) -- (top-4);
\draw[bond] (top-4) -- (top-3);
\draw[bond] (top-0) -- (top-5);
\draw[bond] (top-5) -- (top-3);
\draw[bond] (top-3) -- (top-6);
\draw[bond] (top-6) -- (top-2);
\draw[bond] (top-3) -- (top-7);
\draw[bond] (top-7) -- (top-2);
\draw[bond] (top-0) -- (top-9);
\draw[bond] (top-9) -- (top-8);
\draw[bond] (top-0) -- (top-10);
\draw[bond] (top-10) -- (top-8);
\draw[bond] (top-8) -- (top-11);
\draw[bond] (top-11) -- (top-2);
\draw[bond] (top-8) -- (top-12);
\draw[bond] (top-12) -- (top-2);
\draw[bond] (top-2) -- (top-14);
\draw[bond] (top-14) -- (top-13);
\draw[bond] (top-2) -- (top-15);
\draw[bond] (top-15) -- (top-13);
\draw[bond] (top-13) -- (top-16);
\draw[bond] (top-16) -- (top-1);
\draw[bond] (top-13) -- (top-17);
\draw[bond] (top-17) -- (top-1);
\draw[bond] (top-2) -- (top-19);
\draw[bond] (top-19) -- (top-18);
\draw[bond] (top-2) -- (top-20);
\draw[bond] (top-20) -- (top-18);
\draw[bond] (top-18) -- (top-21);
\draw[bond] (top-21) -- (top-1);
\draw[bond] (top-18) -- (top-22);
\draw[bond] (top-22) -- (top-1);
\draw[bond] (top-0) -- (top-25);
\draw[bond] (top-25) -- (top-24);
\draw[bond] (top-0) -- (top-26);
\draw[bond] (top-26) -- (top-24);
\draw[bond] (top-24) -- (top-27);
\draw[bond] (top-27) -- (top-23);
\draw[bond] (top-24) -- (top-28);
\draw[bond] (top-28) -- (top-23);
\draw[bond] (top-0) -- (top-30);
\draw[bond] (top-30) -- (top-29);
\draw[bond] (top-0) -- (top-31);
\draw[bond] (top-31) -- (top-29);
\draw[bond] (top-29) -- (top-32);
\draw[bond] (top-32) -- (top-23);
\draw[bond] (top-29) -- (top-33);
\draw[bond] (top-33) -- (top-23);
\draw[bond] (top-23) -- (top-35);
\draw[bond] (top-35) -- (top-34);
\draw[bond] (top-23) -- (top-36);
\draw[bond] (top-36) -- (top-34);
\draw[bond] (top-34) -- (top-37);
\draw[bond] (top-37) -- (top-1);
\draw[bond] (top-34) -- (top-38);
\draw[bond] (top-38) -- (top-1);
\draw[bond] (top-23) -- (top-40);
\draw[bond] (top-40) -- (top-39);
\draw[bond] (top-23) -- (top-41);
\draw[bond] (top-41) -- (top-39);
\draw[bond] (top-39) -- (top-42);
\draw[bond] (top-42) -- (top-1);
\draw[bond] (top-39) -- (top-43);
\draw[bond] (top-43) -- (top-1);
\foreach \i in {0,...,43} \node[site] at (top-\i) {};
\node[anchor=east,font=\large] at (-0.160000,5.600000) {$a$};
\node[anchor=west,font=\large] at (12.160000,5.600000) {$z$};
\node[font=\Large] at (6.000000,8.200000) {$D_3^{2,2}$};
% D_1: 4 vertices, 4 edges
\coordinate (before-0) at (0.000000,0.000000);
\coordinate (before-1) at (4.100000,0.000000);
\coordinate (before-2) at (2.050000,1.150000);
\coordinate (before-3) at (2.050000,-1.150000);
\draw[bond] (before-0) -- (before-2);
\draw[bond] (before-2) -- (before-1);
\draw[bond] (before-0) -- (before-3);
\draw[bond] (before-3) -- (before-1);
\foreach \i in {0,...,3} \node[site] at (before-\i) {};
\node[anchor=east,font=\large] at (-0.160000,0.000000) {$a$};
\node[anchor=west,font=\large] at (4.260000,0.000000) {$z$};
\node[font=\Large] at (2.050000,1.800000) {$D_1^{2,2}$};
% D_2: 12 vertices, 16 edges
\coordinate (after-0) at (7.000000,0.000000);
\coordinate (after-1) at (12.000000,0.000000);
\coordinate (after-2) at (9.500000,1.150000);
\coordinate (after-3) at (8.250000,0.897000);
\coordinate (after-4) at (8.250000,0.253000);
\coordinate (after-5) at (10.750000,0.897000);
\coordinate (after-6) at (10.750000,0.253000);
\coordinate (after-7) at (9.500000,-1.150000);
\coordinate (after-8) at (8.250000,-0.253000);
\coordinate (after-9) at (8.250000,-0.897000);
\coordinate (after-10) at (10.750000,-0.253000);
\coordinate (after-11) at (10.750000,-0.897000);
\draw[bond] (after-0) -- (after-3);
\draw[bond] (after-3) -- (after-2);
\draw[bond] (after-0) -- (after-4);
\draw[bond] (after-4) -- (after-2);
\draw[bond] (after-2) -- (after-5);
\draw[bond] (after-5) -- (after-1);
\draw[bond] (after-2) -- (after-6);
\draw[bond] (after-6) -- (after-1);
\draw[bond] (after-0) -- (after-8);
\draw[bond] (after-8) -- (after-7);
\draw[bond] (after-0) -- (after-9);
\draw[bond] (after-9) -- (after-7);
\draw[bond] (after-7) -- (after-10);
\draw[bond] (after-10) -- (after-1);
\draw[bond] (after-7) -- (after-11);
\draw[bond] (after-11) -- (after-1);
\foreach \i in {0,...,11} \node[site] at (after-\i) {};
\node[anchor=east,font=\large] at (6.900000,0.000000) {$a$};
\node[anchor=west,font=\large] at (12.160000,0.000000) {$z$};
\node[font=\Large] at (9.500000,1.800000) {$D_2^{2,2}$};
\draw[-{Stealth[length=3mm]},line width=.8pt] (4.85,0) -- (6.30,0);
\end{tikzpicture}
\caption{The lattice $D_3^{2,2}$ on top and a step of the recursive construction at the bottom.}\label{fig:MKlattice}
\end{figure}

{Writing $D_n^{b,\,s}=(V_n,\,E_n)$, we have}
\[
|E_n|=(bs)^n\approx |V_n|
\]
with the distance between $a$ and $z$ growing as $s^n$. From the above one can derive the effective dimension $d_f=\log(bs)/\log(s)$. On the one hand, this construction reproduces the relation between volume and length associated with Euclidean lattices, making them comparable (for instance for $b=s=2$ and $\mathbb{R}^2$ as in~\cite{hambly2010diffusion}). On the other, it allows non-integer dimensions. Besides, these graphs differ from $\mathbb Z^d$: they lack translation invariance, have a hierarchical arrangement of loops, and possess degrees that grow with the age of a vertex. {Migdal--Kadanoff} lattices provide tractable settings for studying critical phenomena in statistical mechanics, from percolation to diffusive processes~\citep{alves2026percolation,hambly2010diffusion,Clark2022Continuum,Clark2021WeakDisorder,GiacominLacoinToninelli2010Marginal,GiacominLacoinToninelli2010Hierarchical}.

\section{Main results and proof ideas}
\subsection*{The avalanche size} The first result gives matching upper and lower bounds on the tail exponent of an avalanche of the ASM on $D_n^{b,s}=(V_n,\,E_n)$ {with sink set $\partial:=\{a,z\}$. We write $G_n:=D_n^{b,s}/\partial$ for the graph obtained by wiring this boundary into a single sink, also denoted by $\partial$, retaining parallel edges and deleting loops. The notation $D_n^{b,s}$ and $d_n$ always refers to the original, unwired graph.} We add a grain at a selected pivotal vertex $p_j$ on a fixed $a$-to-$z$ geodesic, chosen so that $d_n(a,p_j)=s^j$ for every $n>j$. We keep $j$ fixed as $n$ tends to infinity. The source consequently remains at a fixed distance from the sink (see Figure~\ref{fig:p1} for an example of a pivotal point). 
\begin{figure}[ht!]
\centering\includegraphics[scale=.4]{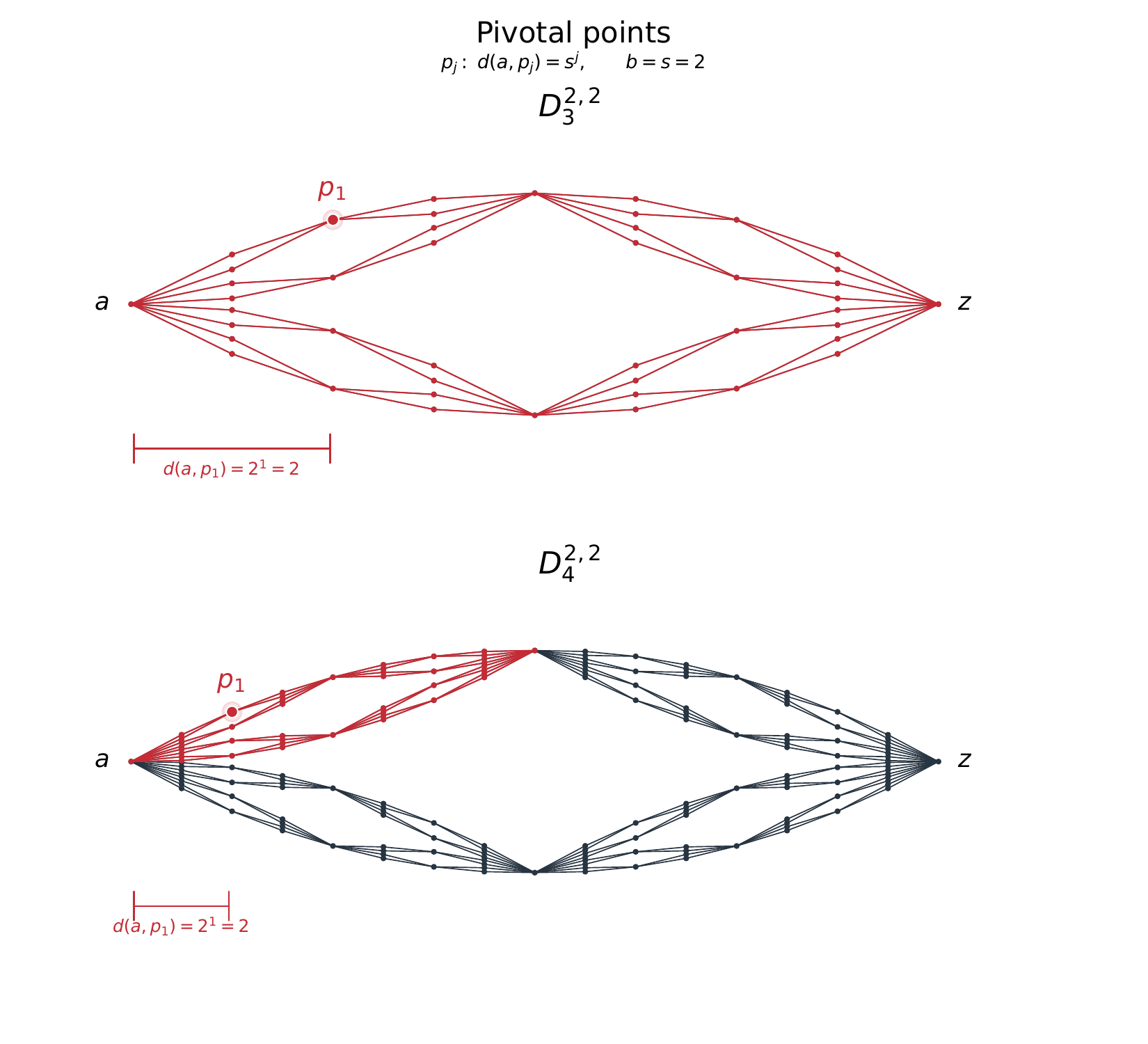}
    \caption{Pivotal point $p_1$ in $D_3^{2,2}$ and $D_4^{2,2}$.}
    \label{fig:p1}
\end{figure}
All estimates are established directly on the finite graphs, without assuming the existence of an infinite-volume sandpile model.

{Let $\nu_n:=\nu_{G_n}$ be the uniform recurrent sandpile measure on $G_n$. Equivalently, this is the ASM on $D_n^{b,s}$ with sink set $\partial$.} \begin{thm}[Power-law for the avalanche]\label{thm:powerlaw}

    Fix $b,\,s\in \{2,\,\ldots\}$ and $j\in \{1,\,2,\,\ldots\}$ to identify $p_j \in \bigcup_{n>j} D_n^{b,\,s}$ such that $d_n(a,\,p_j)=s^j$ for all $n\in \{j+1,\,j+2,\,\dots\}$. There exist two constants $\overline{K}>\underline{K}>0$ {that depend on $b,s$ and $j$ such that}, for all $t\in \left(s\left(\abs{V_j}-1\right),s\left(\abs{V_{n-3}}-1\right)\right]$ and for all $n\geq j+4$,
    \begin{equation*}
        \underline{K} \,t^{-\alpha}<\nu_n\left(\abs{{Av}^{p_j}}>t\right)< \overline{K} \,t^{-\alpha} ,
  \end{equation*}
      with $\alpha:=\log b/\log(bs)$.
\end{thm}
The restriction of the values of $t$ deserves an explanation. At scales smaller that $|V_j|$ we do not expect to see a power-law behaviour. Indeed, by looking at the subgraph isomorphic to $D_j^{b,\,s}$ enclosed by $a$ and $p_j$, we expect the tree rooted at $p_j$ to have a size comparable to $|V_j|$ with substantial probability (we can even say more in Remark~\ref{rem:infinite_series_Q}). On the other hand, the threshold $s\left(\abs{V_{n-3}}-1\right)$ corresponds, asymptotically, to the volume fraction
\[
\frac{s\left(\abs{V_{n-3}}-1\right)}{|V_n|}\approx\frac{s}{(bs)^3}.
\]
This is the size at which our methods begin to detect the finiteness of the system. Finally, the requirement that $n\geq j+4$ is a technical constraint which does not affect our result qualitatively, since we want to look at very large graphs for fixed $j$.

We also prove matching bounds with the same exponent for a uniform wave with source $p_j$. Here a uniform wave is sampled by assigning equal weight to each wave occurrence across all recurrent configurations. {For a nonsink vertex $x$, write $G_n^x:=G_n/\{\partial,x\}$ and $\mathbb{P}_n^x:=\mathbb{P}_{G_n^x}$. Through the wave--forest correspondence recalled below, we also use $\mathbb{P}_n^x$ for the uniform-wave law at $x$.} \begin{thm}[{Power-law} for the uniform wave]\label{thm:powerlaw_wave}

    Fix $b,\,s\in \{2,\,\ldots\}$ and $j\in \{1,\,2,\,\ldots\}$ to identify $p_j \in \bigcup_{n>j} D_n^{b,\,s}$ such that $d_n(a,\,p_j)=s^j$ for all $n\in \{j+1,\,j+2,\ldots\}$. There exist two constants ${K}_u>{K}_\ell>0$ that depend on $b,s$ and $j$ such that, for all $t\in \left(s\left(\abs{V_j}-1\right),\,s\left(\abs{V_{n-3}}-1\right)\right]$,
    \begin{equation*}
        {K}_\ell \,t^{-\alpha}< {\mathbb{P}_n^{p_j}}\left(\abs{{W}^{p_j}}>t\right)< {K}_u\, t^{-\alpha} ,
  \end{equation*}
    {with $\alpha:=\log b/\log(bs)$ and ${\mathbb{P}_n^{p_j}}$ denoting the law} of a uniform wave started at $p_j$.
\end{thm} 

\paragraph{Heuristics and proof idea} Avalanches can be decomposed into waves: the source topples once, the other vertices {stabilise} while the source is prevented from toppling again, and this procedure repeats if needed. The wave has the law of the component containing \(p_j\) in a uniform spanning forest with two components rooted at \(p_j\) and the sink~\citep{IKP}.  This makes random walk hitting probabilities available for estimating how far a wave reaches.

The key estimates concern the pivotal vertices. For $j\leq k<n$, the probability that a uniform wave reaches $p_k$ is comparable to $b^{j-k}$. A level-$k$ subgraph has volume of order $(bs)^k$. We are able to relate both quantities to the {wave size and, once this link is established,} the change of variable $t\asymp(bs)^k$ gives

\[
\nu_n\left(\abs{{Av}^{p_j}}>t\right)\approx b^{-k}=\left((bs)^k\right)^{-\frac{\log b}{\log(bs)}}.
\]
For fixed $j$, the factor $b^j$ is absorbed into the comparison constants. This gives a heuristic argument for the exponent.  
For uniform waves, the lower bound will be proved in Section~\ref{sec:LB_uniform}. To pass to the {first-wave}, we need Lemma~\ref{lm:UB by tree}, which shows that the {first-wave} measure is a perturbation of the uniform wave measure. The strength of this perturbation can be controlled quantitatively: indeed, it can be computed using the estimates of Subsection~\ref{subsec:Green} on the Green's function of the simple random walk on the graph at the pivotal points. We are able to calculate its value exactly, showing that it is suitable for use in a first-moment argument. Finally, the avalanche tail bounds are obtained, on the lower side, as a consequence of the first-wave bound (since the first wave is obviously contained in the avalanche); on the upper side, by using our knowledge of the Green's function, that is, the estimate on the expected number of topplings at $p_j$. 
\subsection*{Pivotal reach and a second boundary setup} In this paper we also address {the limiting behaviour of two ASM observables}. The first is the ``pivotal reach'', that is, the farthest pivotal point from the source that the avalanche reaches. 
In Section~\ref{sec:R'} we show that the pivotal reach admits an almost surely finite weak limit as the graph size tends to infinity. This is {characterised} in the following result.
\begin{prop}\label{prop:R_intro_res}
    Let 
    \begin{equation*}
    R_n:=\max\left(\left\{k:j\leq k<n,\,p_k\in Av^{p_j}\right\}\cup \{0\}\right).
\end{equation*}
Then, at least along a subsequence, there exists an almost surely finite and {nonconstant} random variable $R$ such that $\lim_{n\to \infty}R_n=R$ in law.
\end{prop}
We also provide an additional result of the same {flavour} for the pivotal reach of a uniform wave in Lemma~\ref{lemma:limit_L}.

Our second main result concerns a different choice of boundary conditions. {We now take $a$ as the sole sink on the unwired graph $D_n^{b,s}$ and consider a uniform wave with source $z$, whose degree is $b^n$. Wiring this source to the sink gives $G_n=D_n^{b,s}/\partial$; the endpoint-wave forest is therefore sampled under $\mathbb{P}_{G_n}$.} Let $\mathrm{C}_n$ denote the number of vertices in this wave. Its expected size is asymptotically
half the graph volume, and its {normalisation} by the mean converges in distribution:
\begin{thm}\label{thm:limit Dbs}
    {For all $t\in \mathbb{R}$, define}
    \begin{equation*}
        f(t)\coloneq \sum_{i=1}^{\infty}b^i\log\left(\dfrac{1}{s}\sum_{k=0}^{s-1}\left(e^{2t(bs)^{-i}}\right)^k\right).
    \end{equation*}
     Let $$Q_n=\frac{\mathrm{C}_n}{\mathrm{E}[\mathrm{C}_n]}.$$
    Then $f(t)$ converges for all $t\in \mathbb{R}$ and it holds that
    \begin{equation*}
        \lim_{n\to\infty}\mathrm{E}\left[e^{tQ_n}\right]=e^{f(t)} \quad \forall \,t \in \mathbb{R}.
    \end{equation*}
    In particular, $Q_n$ converges in distribution to a nondegenerate random variable $Q$ whose law is uniquely determined by $e^{f(\cdot)}$ as {its moment-generating function}.
\end{thm}  
\paragraph{Proof idea}
{For the result on $R_n$, we} exploit the recursive structure of the hierarchical lattice and the estimates for the simple random walk. The bounds obtained for the avalanche size are also suitable for controlling the growth of $R_n$, whose law follows from a recursive distributional decomposition of the forest component. This
endpoint regime complements the tail estimates near a sink: it describes waves whose
typical size is of the same order as the entire graph. 
As for $Q_n$, we find a sequence of {\em junction vertices} that control the size of the wave started at $z$. Basically, they are ``bottlenecks'' that the wave has to cross in order to grow in size: knowing they are reached allows one to track the radius of the wave in terms of spanning trees, which are more tractable.

\subsection*{{Structure of the paper}} In {Section~\ref{sec:prelim}, we set} up the notation and recall the basic facts we will use about the random processes we need. Theorem~\ref{thm:powerlaw_wave} {is proved in two sections}: Section~\ref{sec:UB_uniform} proves the upper bound, while the lower bound is shown in Section~\ref{sec:LB_uniform}. The avalanche {power-law} is shown in Section~\ref{sec:proofAv}. Section~\ref{sec:R'} is dedicated to the proof of Proposition~\ref{prop:R_intro_res}, and {the final section, Section~\ref{sec:Cb}, deals} with the limit law of the size of the uniform wave at $z$.
\subsection*{{Use of generative AI}} With the exception of Lemma~\ref{lemma:limit_L} and {Remark~\ref{rem:infinite_series_Q}, which were suggested by ChatGPT-6 Astra during the final review stage, all ideas and proofs in this paper are formulated and developed by the authors without {any use of generative AI}. {Elsewhere, ChatGPT-6 was used for proofreading, converting the authors' hand-drawn figures into portable format and improving the flow of the introduction.} The authors take full responsibility
for the mathematical content and presentation of the paper.
\section{{Notation and preliminaries}}\label{sec:prelim}
\subsection{Notation}
We use the notation $\mathbb{N}\coloneq\{0,\,1,\,2\,\ldots\}$ and $\mathbb{N}_+\coloneq\{1,\,2\,\ldots\}$. 

{For two real numbers $x,\,y$, we denote their maximum and minimum by $x\vee y$ and $x\wedge y$, respectively.}

{Sums, products and unions or intersections over an empty index set are defined to be 0, 1 and $\emptyset$, respectively.}
The natural logarithm is denoted by $\log$, and we will make it explicit if we are using another base.

We work with finite, connected, undirected multigraphs without self-loops. Parallel edges are retained as distinct edges. We write $G=(V,E)$, use $V(G)$ and $E(G)$ when needed, and denote by $m_G(u,v)$ the number of edges between $u$ and $v$.

For a nonempty set $S\subseteq V(G)$, write $G/S$ for the graph obtained by identifying all vertices of $S$ as a single vertex $\partial$. Every original edge with at least one endpoint outside $S$ is retained, with its endpoints mapped to the quotient; edges with both endpoints in $S$ are deleted. Thus parallel edges and their multiplicities are preserved, and
\[
 V(G/S)=(V(G)\setminus S)\cup\{\partial\},\qquad
 \deg_{G/S}(v)=\deg_G(v)\quad(v\notin S).
\]
We use $\partial$ for the sink vertex throughout. When a graph already has sink $\partial$, we write
\[
 G^x:=G/\{\partial,x\}
\]
for the graph in which the nonsink vertex $x$ is also wired to the sink.

{We denote by $X$ the random walk on $G=(V,E)$ that chooses an incident edge uniformly at each step, so its transition probability from $u$ to $v$ is $m_G(u,v)/\deg_G(u)$} and display its starting point $X_0=x$ on the associated probability measure $\PRW^x$. We denote by $\tau_x$ {the first hitting time in $x$ of $X$}, that is,
\begin{equation*}    \tau_x\coloneq{\min}\{k\geq 0: \,X_k=x\}.
\end{equation*}
{For a nonempty vertex set $S$, let $\tau_S:=\min\{k\geq0:X_k\in S\}$.}
{For every $x,y \in V$, we} {denote by} $d_{G}(x,y)$ {the graph distance between the two vertices} $x$ and $y$, defined as {the minimum number of edges in a path linking the two vertices in $G$}. Denote by $\sim$ {the binary relation on $V$} of being at distance one with respect to $d_G$. We will {use the shorthand} $d_n(\cdot,\cdot)$ to denote the graph distance on a sequence of graphs $D_n^{b,\,s}$ defined in Section~\ref{subsec:hilattice}. 

If $f(\cdot,\,\cdot)$ is a function of {two variables, we use the shorthand} $f(x):=f(x,x)$. We use $\log$ to denote the natural logarithm.

Given two sequences $(a_n)_{n\in \mathbb{N}}$ and $(b_n)_{n\in \mathbb{N}}$ of real numbers, {we write $a_n \asymp b_n$} if
\begin{equation*}
    \lim_{n\xrightarrow {}\infty }\dfrac{a_n}{b_n}\in \mathbb{R}\backslash \{0\}.
\end{equation*}
The notation $a_n=o_n(b_n)$ is used when { $\lim_{n\to\infty}a_n/b_n=0$.}
%Given two probability measures $\bbP_1$ and $\bbP_2$ we denote by $\bbP_1 \otimes \bbP_2$ the corresponding product measure. 

\subsection{Spanning trees}\label{sec:SpanningT}
Here we recall some definitions concerning spanning trees and {fix notation for Section}~\ref{subsec:ASM}.

Given a connected finite graph  $G=(V,E)$ and a vertex $x\in V$, a  \emph{rooted spanning tree} is an acyclic and connected subgraph of $G$ whose vertex set equals $V$ and with a ``special'' vertex $x$ marked as the root of the tree. Denote {the set of such spanning trees by} $\Omega_{G}$, omitting the choice of {the root, which in our case will always coincide with the sink} of the ASM. Equip $\Omega_G$  with the uniform probability measure $\bbP_G$.

Similarly, a \emph{two-component spanning forest} rooted at distinct vertices $x,y\in V$ is a pair $(T^x,T^y)$ of disjoint trees whose vertex sets partition $V$, with $x\in T^x$ and $y\in T^y$. Identifying the two roots maps such a forest bijectively to a spanning tree of $G/\{x,y\}$. Edges retain their original identities, so the two components are recovered by separating the roots again. Accordingly, under $\mathbb{P}_{G/\{x,y\}}$, the notation $T^x$ denotes the component rooted at $x$ in this forest on the original graph. For a graph with sink $\partial$, the same convention applies under $\mathbb{P}_{G^x}$ to the forest $(T^x,T^{\partial})$ on $G$.

%Commnad-shift-7 to comment everyhting out

%

\subsubsection{The groundskeeper algorithm} Fix $x\in V $ and consider the following sampling procedure in $\Omega_G$, producing a uniform spanning tree $T$. We start with $T=\emptyset$. Choose a vertex $v\in V$ and start a simple random walk on $G$ with $X_0=v$. The first time that a vertex is visited through an edge $e$, we add $e$ to the edges of the spanning tree. Once $G$ is covered, the resulting set of edges forms a spanning tree sampled uniformly at random, independently of the starting position of the walk. The procedure just described is usually referred to as the groundskeeper, or Aldous--Broder, algorithm~\citep{aldous1990random,broder1989generating}. Other important sampling algorithms exist (for instance~\cite{wilson1996generating}).

Precisely, our main application of the groundskeeper algorithm will be the following one:
\begin{prop}\label{prop:GK}
{Let $x,y,w\in V$ be three distinct vertices in $G$. Then}
    \begin{equation*}      {\mathbb{P}_{G/\{x,y\}}}\left(w\in T^x\right)={\mathcal{P}^w_G\left(\tau_x<\tau_y\right).}
    \end{equation*}
\end{prop}
\begin{proof}
    {This follows from the GA or an equivalent algorithm (see, for example,}~\citep[Eq. (2.7)]{bourabee2026sandpileavalanchesdimensionsfinite}).
\end{proof}

\subsection{The Abelian sandpile}\label{subsec:ASM}
{Given a graph $G=(V\cup\{\partial\},E)$ as above, with nonsink vertex set $V$ and sink $\partial$, the Abelian sandpile model is defined as follows} \citep{DHAR1994}.
At each site $v\in V$, we define a \emph{height} variable $h_v\in\mathbb{N}_+$ called the
height of the sandpile, and a threshold value $\eta_v$. If $h_v\leq \eta_v$, for all $v$, the pile is said to be
\emph{stable}. Fix a strictly positive probability measure on $V$. {The time evolution} of the sandpile is defined by the following two rules:
\begin{enumerate}
    \item  {elect one} of the sites according to the probability measure, and add a grain of sand there. On addition of the grain at site $v$, $h_v$ increases by 1. Heights at other sites remain unchanged.
    \item If for a site $v$ one has $h_v> \eta_v$, then $v$ is said to be \emph{unstable}, so it \emph{topples},
and loses some {sand grains} to other sites. This sand grain transfer is defined in terms of a $|V|\times |V|$ integer matrix $\Delta$, by means of which we set $\eta_v\coloneqq \Delta(v,\,v)$. On toppling at site $v$, the configuration is updated according
to the rule
\[
h_{u}\mapsto h_u-\Delta(v,\,u).
\]
If the toppling results in some other site becoming unstable, it is also toppled. Particles
arriving at the sink are lost, that is, we do not keep track of them. The process continues until all sites become stable.
\end{enumerate}

{For $v\in V$, denote} by $a_{v,G}$ the function that takes a stable configuration $h$ and returns the {stabilisation} of $h+\delta_v$.

A common choice for $\Delta$, {which we also adopt}, is the matrix obtained from the {unnormalised} graph Laplacian of $G$ by removing the column and row corresponding to the sink. {Recall that $m_G(u,v)$ is the number of edges between $u$ and $v$.} We define, for $v,\,u\in V$, the Laplacian as
\begin{equation}\label{eq:delta}
\Delta(v,\,u)=\begin{cases}
    \deg_G(v) &\text{if }v=u\\
    {-m_G(u,v)}  & \text{if }\overline{vu}\in E\\
    0 &\text{otherwise}
\end{cases}.
\end{equation}
Note that, since $G$ is connected, it can be shown that the matrix $\Delta$ is invertible. 

It is easy to see that the dynamics of the ASM can be defined as a discrete-time Markov chain on the space of stable configurations. {Eventually, the chain} reaches the unique recurrent class $\mathrm{R}_G$~\citep[Theorem $2.9$]{JaraiSurvey}. {Therefore, studying the properties of the chain in the long run means studying}~$\mathrm{R}_G$, which we equip with the stationary measure $\nu_G$ of the chain over it. This measure is simply the uniform distribution on recurrent configurations~\citep[Exercise 2.16]{JaraiSurvey}. For the many properties of the ASM, including the fact that the toppling order {does not affect the final stable configuration}, we refer the reader to~\cite{JaraiSurvey,redig2006mathematical}, for instance. Here we will recall only the important connections with uniform spanning trees.

The~\citeauthor{majumdar1992equivalence}, or burning, bijection proves that there is a one-to-one correspondence between spanning trees of a graph and $\mathrm{R}_G$. The burning bijection yields an algorithm to check whether a height function belongs to $\mathrm{R}_G$. In particular, a configuration is recurrent if and only if it is burnable, and {the order in which sites burn uniquely reconstructs a spanning tree} (up to a choice of the burning order of the edges). {In turn, this and the matrix-tree theorem} {(see~\citep{biggs}) already yield the important consequence} that $\abs{\mathrm{R}_G}=\det\Delta$. 
There is also an alternative condition to determine whether a certain $h$ is recurrent or not. Given a height configuration $h$, a forbidden sub-configuration (FSC) is defined as a set $F\subset V$ that does not burn in the burning {procedure; in other words, a configuration} for which
\begin{equation*}
    h_{v}\leq \abs{\{u\in F:v \sim u \}}, \, \forall v \in F.
\end{equation*}
Then $h$ is a recurrent configuration if and only if it does not contain {an FSC}, which means that
\begin{equation}\label{FSC cond}
    \forall F\subset V, \, \exists v\in F: \, h_v>\abs{\{u\in F: v\sim u\}},
\end{equation}
with $F\neq \emptyset$.

One interesting observable in the ASM is the avalanche cluster $Av^x\subset V$, composed of all the vertices that topple when adding a grain at $x\in V$. For a start, one may be interested in understanding how many times a particular site topples. For all $u,\,v\in V$, denote by $N(u,v)$ the integer-valued random variable that gives the number of topplings at $v$ when adding {a grain at $u$} (where the randomness is in the starting configuration). As shown by~\citeauthor{dhar}, the expectation of such a random variable at stationarity can be obtained through the inverse of $\Delta$:
\begin{equation}\label{eq:E[N]}
\mathbb{E}_{\nu_G}\left[N(u,\,v)\right]=\left(\Delta^{-1}\right)(u,\,v).
\end{equation}

{Using the Abelian property of the model,~}\citeauthor{IKP} show a convenient way of performing the {stabilisations} in an avalanche $Av^x$. More precisely, after the (possible) first toppling of $x$, one topples all other unstable vertices until $x$ is the only one left unstable. The vertices that have toppled in this procedure compose what is called the \emph{first-wave} $W_1^x$, and if $x$ is unstable one can continue this procedure to construct the second wave $W^x_2$, and so on. Recall that each vertex, except possibly $x$ itself, topples only once in a wave~\citep{JaraiSurvey}. This leads to the following decomposition of the avalanche cluster:
\begin{equation}\label{eq:decomp_wave}
    Av^x=\bigcup_{i=1}^{N(x)}W^x_i.
\end{equation}
Note that waves are not necessarily nested sets: in fact, for the hierarchical lattice as well as other graphs, they are not~\citep{dhar1994inverse}. This means that {for our model, treating single waves} is not sufficient to study the avalanche size. For $k\in \{1,\,\dots,\,N(x)\}$ we {denote by} $\mathcal{W}^x_k$  the operator that acts on configurations $h\in \mathrm{R}_G$ with $h_x=\deg_G(x)$, by yielding the configuration after the $k$-th wave. Note that for $k\neq N(x)$, $\mathcal{W}^x_k$ returns an unstable configuration at $x$.

The decomposition~\eqref{eq:decomp_wave} {reveals another connection} with spanning trees, {also introduced by~\citeauthor{IKP}}, which we now recall. {Fix $x\in V$ and use the source-wired graph $G^x=G/\{\partial,x\}$ defined above.} Every $h\in \mathrm{R}_G$ that can start a first {nonempty} wave after dropping a grain at $x$ (i.e. that has $h_x=\deg_G(x)$) identifies a unique height function $h^w$ in ${G^x}$ by simply removing the $x$-coordinate from $h$. {Moreover, using~}\eqref{FSC cond} it is easy to see that such $h^w$ is in $\mathrm{R}_{{G^x}}$. Hence a restriction to an appropriate subset of $\mathrm{R}_{{G^x}}$ yields the following bijection:
\begin{equation}\label{bij IKP}
    \{h\in \mathrm{R}_G:h_x=\deg_G(x)\}\longleftrightarrow {E^x}\coloneqq\{h^w\in \mathrm{R}_{{G^x}}:(h^w,\deg_G(x))\in \mathrm{R}_G\}.
\end{equation}
Denote by {$\mathcal A:=\prod_{v\in V\setminus\{x\}}a_{v,G^x}^{m_G(x,v)}$} the operator acting on stable configurations {in ${G^x}$ by sending} {$m_G(x,v)$ grains to each nonsink vertex $v\ne x$ and then stabilising; grains sent from $x$ directly to $\partial$ are absorbed}. For each integer $n\geq 2$ we use the convention $\mathcal{A}^n:=\mathcal{A}^{n-1}\circ \mathcal{A}$, with $\mathcal{A}^1=\mathcal{A}$. {Then~\citeauthor{IKP} show} that the following map is a bijection:
\begin{align}\label{bij IKP wave}
    F^x:=\{\left(h,\,k\right): h\in \mathrm{R}_G,\,h_x=\deg_G(x),\,k\in \{1,\,\dots,\,N(x)\} \}&\xrightarrow{}\mathrm{R}_{{G^x}}\nonumber,\\
    \left(h,\,k\right)&\mapsto\mathcal{A}^k\left(h|_{V\backslash \{x\}}\right).
\end{align}
{Their proof relies upon \emph{inverse waves}, a procedure for reconstructing} the whole wave dynamics from the {post-wave} configuration. Each pair $(h,\,k)$ {on the left-hand side} of~\eqref{bij IKP wave} is interpreted as the configuration after the $k$-th wave started from $h$, i.e. $\mathcal{W}^x_k(h)$. The construction of bijection~\eqref{bij IKP wave} {provides further information about the wave clusters} (i.e. the vertices that topple during the wave). Recall that $\mathrm{R}_{{G^x}}$ is in bijection with the set of two-component spanning forests of $G$ {rooted at $x$ and ${\partial}$}, which are pairs of the form $\left(T^x,\,T^{{\partial}}\right)$. Then one can show that the vertex set that topples in the $k$-th wave arising from $h$ is the vertex set of $T^x$, constructed by {applying the burning algorithm to} $\mathcal{A}^k\left(h|_{V\backslash{x}}\right)$.

The two bijections~\eqref{bij IKP} and~\eqref{bij IKP wave} yield the following counting formula for recurrent configurations in $G$ that satisfy $\{\abs{W^x_1}>t\}$ for $t\in \mathbb{N}$:

\begin{equation}\label{eq:A}     \nu_G\left(|W^x_1|> t\right)=\dfrac{\abs{h\in \mathrm{R}_G:\,|W^x_1|> t}}{\abs{\mathrm{R}_G}}=\dfrac{\abs{E^x\cap\{h^w\in\mathrm{R}_{{G^x}}:\,|V\left(T^x\left(\mathcal{A}(h^w)\right)\right)|> t\}}}{\abs{\mathrm{R}_G}},
\end{equation}
where we used that $\{|W^x_1|>t\}\subset \{h_x=\deg_G(x)\}$ and the fact that $T^x$ is a function of $h^w\in \mathrm{R}_{{G^x}}$. This is one of the key identities of the paper, {which enables us to relate} the problem of the size of the first-wave to a problem on spanning forests. The next main task is now to compare the right-hand side of~\eqref{eq:A} with the corresponding event under the uniform two-component forest
measure.

This is what we do in {the next lemma}. It will be used in Sections~\ref{sec:UB_uniform} and~\ref{sec:LB_uniform} to quantify how far off the first-wave is from the connected component of $x$ in a tree distributed according to  $\mathbb{P}_{{G^x}}$.

\begin{lemma}\label{lm:UB by tree}
Fix $x\in V$ and $V'\subset V$ such that $x\in V'$. {Recall that $G^x=G/\{\partial,x\}$.} Then there exists $C>0$ that depends on $d_G({\partial},\,x)$ and $\max_{y\in V}\deg_G(y)$ such that
\begin{equation}\label{eq:lm1}
    \dfrac{1}{C}\mathbb{P}_{{G^x}}\left(V\left(T^x\right)\nsubset V' \right)\dfrac{\abs {\mathrm{R}_{{G^x}}}}{\abs{\mathrm{R}_{G}}}\leq \nu_G\left(W^x_1\nsubset V' \right)\leq \dfrac{\abs {\mathrm{R}_{{G^x}}}}{\abs{\mathrm{R}_{G}}}\mathbb{P}_{{G^x}}\left(V\left(T^x\right)\nsubset V'\right).
\end{equation}
{Moreover, the same upper bound also holds} for the other waves:
\begin{equation}\label{eq:UB_l_wave}
    \nu_G\left(W^x_\ell\nsubset V'\right)\leq \dfrac{\abs{\mathrm{R}_{{G^x}}}}{\abs{R_G}}\mathbb{P}_{{G^x}}\left(V\left(T^x\right)\nsubset V'\right),
\end{equation}
{for all $\ell\in \mathbb{N}_+$, with the convention} that $W^x_{\ell}=\emptyset$ for $\ell>N(x)$.
\end{lemma}
\begin{proof}
    We start by proving the result for the first-wave, starting from the upper bound of~\eqref{eq:lm1}. {If $W_1^x=\emptyset$, then} the result is trivial, so we now consider the case $W_1^x\neq \emptyset$, which occurs {if and only if $h_x=\deg_G(x)$.}

    Given $h^w\in \mathrm{R}_{{G^x}}$, denote by $\mathcal{T}(h^w)$ the set of vertices that topple when applying $\mathcal{A}$.
    
    Recall that the first-wave in $G$ can be simulated in the wired graph since we know that \mbox{$W^x_1(h)=\{x\}\cup \mathcal{T}(h|_{V\backslash \{x\}})$}. Note that $h|_{V\backslash \{x\}}$ is {a recurrent configuration} in ${G^x}$. We can use the bijection~\eqref{bij IKP} with configurations in $E^x\subset \mathrm{R}_{{G^x}}$:
    \begin{equation*}
    \abs{h\in \mathrm{R}_G:\{h_x=\deg_G(x)\} \cap\{W^x_1(h)\nsubset V'\}}=\abs{E^x\cap\left\{h^w\in\mathrm{R}_{{G^x}}:\mathcal{T}(h^w)\nsubset V'\right\}}.
    \end{equation*}
    {By definition of the addition operators, ${\mathcal A h^w}\in \mathrm{R}_{{G^x}}$; hence it is associated} with a unique spanning tree of ${G^x}$, or equivalently a two-component spanning forest $\left(T^x,\,T^{{\partial}}\right)$. Using the~\citeauthor{IKP} bijection, one sees that the vertices that topple in the above {stabilisation} procedure are the ones in $T^x$:
    \begin{equation}
       \abs{E^x\cap\left\{h^w\in \mathrm{R}_{{G^x}}:\mathcal{T}(h^w)\nsubset V'\right\}}\leq \abs{h^w\in\mathrm{R}_{{G^x}}:\mathcal{T}(h^w)\nsubset V'}=\abs{\Omega_{{G^x}}:V\left(T^x\right) \nsubset V'}.\label{eq:T_G_V}
    \end{equation}
    Translating everything {in terms of the uniform measure} and using that $\abs{\Omega_G}=\abs{\mathrm{R}_G}$, we get the desired result:
    \begin{equation*}
\begin{aligned}
        \nu_G\left(W^x_1\nsubset V'\right)
        &=\dfrac{\abs{h\in \mathrm{R}_G:W^x_1(h)\nsubset V'}}{\abs{\mathrm{R}_G}}\\
        &\leq \dfrac{\abs {\mathrm{R}_{{G^x}}}}{\abs{\mathrm{R}_{G}}}\dfrac{\abs{\Omega_{{G^x}}:V\left(T^x\right)\nsubset V'}}{\abs {\mathrm{R}_{{G^x}}}}\\
        &=\dfrac{\abs {\mathrm{R}_{{G^x}}}}{\abs{\mathrm{R}_{G}}}\mathbb{P}_{{G^x}}\left(V\left(T^x\right)\nsubset V'\right).
\end{aligned}
    \end{equation*}
    We now prove the lower bound of~\eqref{eq:lm1}. The main issue is that the configurations in $E^x$ do not have a useful representation in terms of uniform spanning trees. Thus we introduce a new way to generate configurations in $E^x$.

Firstly, {choose a geodesic path $\gamma$ in $G$ from $\partial$ to $x$, including both endpoints.} Then, given a configuration $h^w\in \mathrm{R}_{{G^x}}$, denote by $\mathrm{L}$ the \emph{lifting operator} $\mathrm{R}_{{G^x}}\to \mathrm{R}_{{G^x}}$ that assigns maximal height to the internal vertices of $\gamma$:
\begin{equation*}
    \mathrm{L}(h^w):=\begin{cases}
        \deg_G(v),\quad  &v\in \gamma\backslash \{x,\,{\partial}\},\\
        h^w_v, \quad &v\notin \gamma. 
    \end{cases}
\end{equation*}
The fact that $\mathrm{L}(h^w)\in \mathrm{R}_{{G^x}}$ follows immediately from~\eqref{FSC cond}. We claim the following inclusion:
\begin{equation}\label{eq:lm1_lb}
\mathrm{L}\left(\mathrm{R}_{{G^x}}:\mathcal{T}\left(h^w\right)\nsubset V'\right){\subset}E^x\cap \left\{\mathcal{T}\left(h^w\right)\nsubset V'\right\}.
\end{equation}
{We now prove this inclusion.} In fact, take $h^w\in \mathrm{R}_{{G^x}}$ and consider  $\left(\mathrm{L}\left(h^w\right),\,\deg_G(x)\right)$, which we argue is recurrent in $G$. Indeed, {by assigning maximal height to the vertices of $\gamma$, the fire} propagates from ${\partial}$ to $x$ at least  through the vertices in $\gamma$. Then the rest of the vertices will be reached by fire since $\mathrm{L}(h^w)\in \mathrm{R}_{{G^x}}$. This proves that $\mathrm{L}\left(\mathrm{R}_{{G^x}}\right) \subset E^x$. To conclude the proof of~\eqref{eq:lm1_lb}, note that $\mathcal{T}\left(h^w\right)\subset \mathcal{T}\left(\mathrm{L}(h^w)\right)$, since increasing the heights can only create new topplings.  

{With~\eqref{eq:lm1_lb} in hand, it suffices} to find a lower bound for $\abs{\mathrm{L}\left(\mathrm{R}_{{G^x}}:\mathcal{T}\left(h^w\right)\nsubset V'\right)}$. This can be done by exploiting the properties of the lifting operator. In fact, the maximum number of elements in the preimage $\mathrm{L}^{-1}(h^w),\,h^w\in \mathrm{R}_{{G^x}}$, is bounded by the number of stable (and not necessarily recurrent) configurations in $\gamma$, which is {bounded above by}
\begin{equation}\label{eq:stima_C}
    C:=\prod_{v\in \gamma\backslash \{x,\,{\partial}\}}\deg_G(v)\leq \max_{v\in \gamma \backslash \{x,\,{\partial}\}}(\deg_G(v))^{d_G\left(x,\,{\partial} \right)-1}.
\end{equation}
Thus, since each $h^w\in \mathrm{R}_{{G^x}}$ can be reached through $\mathrm{L}$ by at most $C$ stable height {configurations, we get that}
\begin{equation*}
    \abs{\mathrm{L}\left(\mathrm{R}_{{G^x}}:\mathcal{T}\left(h^w\right)\nsubset V'\right)}\geq \dfrac{\abs{\left \{\mathrm{R}_{{G^x}}:{\mathcal{T}\left(h^w\right)\nsubset V'}\right \}}}{C},
\end{equation*}
thus proving the following lower bound: 
\begin{equation}\label{ineq:LB1}
   \abs{E^x\cap\left\{\mathrm{R}_{{G^x}}:\mathcal{T}(h^w)\nsubset V'\right\}}\geq \dfrac{\abs{\{\mathrm{R}_{{G^x}}:\mathcal{T}(h^w)\nsubset V'\}}}{C}.
\end{equation}
From here we proceed as in~\eqref{eq:T_G_V}.

{Finally, we prove}~\eqref{eq:UB_l_wave}. Fix now an integer $\ell\geq 2$. Again, we would like to simulate the $\ell$-th wave in ${G^x}$ using that  $W^x_{\ell}(h)=\{x\}\cup \mathcal{T}\left(\mathcal{A}^{\ell-1}h\right)$. The only problem is that if $\ell>N(x)$, then $W^x_{\ell}(h)=\emptyset$ but it is not necessarily true that $\mathcal{T}\left(\mathcal{A}^{\ell-1}h\right)=\emptyset$. {In other words, we can work} in ${G^x}$, forgetting the height configuration in $x$, only if we know a priori that the $\ell$-th wave is actually {nonempty}.

    In order for the $\ell$-th wave to occur, we have to require that when applying $\left(\mathcal{A}^i\right)_{i\in \{0,\,\dots ,\, \ell-2\}}$ the nearest neighbours of $x$ topple:
    \begin{equation*}
        \abs{\mathrm{R}_G:W^x_{\ell}(h)\neq \emptyset}=\abs{E^x\cap \left\{\bigcap_{i=1}^{\ell-1}\left\{ \{y:y\sim x\}\subset \mathcal{T}\left(\mathcal{A}^{i-1}h^w\right)\right \} \right \}}.
    \end{equation*}
    Hence the correct way to translate the problem from $G$ to ${G^x}$ {is obtained through} the following equality:
    \begin{align}\label{eq:wi<w1}
        \abs{\mathrm{R}_G:W^x_\ell(h)\nsubset V'}&=\abs{E^x \cap \left\{\bigcap_{i=1}^{\ell-1}\left\{  \{y:y\sim x\}\subset \mathcal{T}\left(\mathcal{A}^{i-1}h^w\right)  \right \}\right \}\cap \{\mathcal{T}\left(\mathcal{A}^{\ell-1} h^w\right)\nsubset V'\}}\nonumber\\
        &\leq \abs{\mathrm{R}_{{G^x}}:\mathcal{T}\left(\mathcal{A}^{\ell-1} h^w\right)\nsubset V'}.
    \end{align}
   Since each addition operator is a bijection on the set of recurrent configurations, so is $\mathcal{A}^{\ell-1}$ as a composition of those operations. Therefore
    \begin{equation*} \abs{\mathrm{R}_{{G^x}}:\mathcal{T}\left(\mathcal{A}^{\ell-1}h^w\right)\nsubset V'}=\abs{\mathrm{R}_{{G^x}}:\mathcal{T}\left(h^w\right)\nsubset V'},
    \end{equation*}
    and now one can {proceed} as in~\eqref{eq:T_G_V} to conclude.
\end{proof}    
\begin{remark}\label{rem:larger_t}
Equations~\eqref{eq:T_G_V} and~\eqref{eq:wi<w1} {also hold for events} of the form $\{|W_\ell^x|>t\}$, and therefore Lemma~\ref{lm:UB by tree} also holds for such events, with the appropriate replacements on the left- and right-hand sides.
\end{remark}
\begin{remark}\label{rem:link_with_uni_wave}
    {By the~}\citeauthor{IKP} bijection, $\mathbb{P}_{{G^x}}$ can be interpreted as {the uniform measure on wave configurations}. {More precisely, for each} {nonempty} $V'\subset V$ it holds that
    \begin{align*}
        \mathbb{P}_{{G^x}}\left(V\left(T^x\right)\nsubset V'\right)&=\dfrac{\abs{T\in \Omega_{{G^x}}:V\left(T^x\right)\nsubset V'}}{\abs{\Omega_{{G^x}}}}\\
    &=\dfrac{\abs{(h,\,k)\in F^x:\mathcal{T}\left(\mathcal{A}^{k}h|_{V\backslash \{x\}}\right)\nsubset V'}}{\abs{F^x}}.\\       
    \end{align*}
    Hence Lemma~\ref{lm:UB by tree} allows us to control the behaviour of the first-wave through {a uniformly chosen wave}.
\end{remark}

In order to make use of {Lemma~\ref{lm:UB by tree}, one needs to find an upper bound for the normalisation factor} $\abs{\mathrm{R}_{{G^x}}}/\abs{\mathrm{R}_{G}}$. Such a ratio can be shown to be equal to the expected number of waves with source $x$. 
\begin{corollary}
  {For $x\in V$, it holds} that
\begin{equation}\label{eq:density=E[N]}    \dfrac{\abs{\mathrm{R}_{{G^x}}}}{\abs{\mathrm{R}_{G}}}=\mathbb{E}_{\nu_G}\left[N(x)\right].
\end{equation}  
\end{corollary}
\begin{proof} Using Dhar's bijection and {the matrix-tree theorem}, we can express the {normalisation} in terms of {the inverse matrix $\Delta^{-1}$}, where $\Delta$ is defined in~\eqref{eq:delta}: 
\begin{equation*}
    \dfrac{\abs{\mathrm{R}_{{G^x}}}}{\abs{\mathrm{R}_{G}}}=\dfrac{\abs{\Omega_{{G^x}}}}{\abs{\Omega_{G}}}=\dfrac{\det \Delta_{x}}{\det \Delta}=\left(\Delta^{-1}\right)(x,\,x),
\end{equation*}
where we denoted by $\Delta_x$ the matrix obtained from $\Delta$ by removing its $x$-th column and row and used Cramer's rule to compute determinants. Hence~\eqref{eq:E[N]} yields the {desired} result.
\end{proof}
\subsection{The hierarchical lattice}\label{subsec:hilattice} We are now going to define the family of graphs on which we will study the ASM: the \emph{hierarchical lattices}.

Fix a \emph{branching} number $b\in \mathbb{N},\,b\geq 2$, and a \emph{segmenting} number $s\in \mathbb{N},\,s\geq 2$. From this point onwards in the paper, $b$ and $s$ will be fixed and therefore {omitted from the notation for the sake of exposition}. Define the starting graph to be composed of two vertices $a$ and $z$ connected by {an undirected edge}, so that $D_0\coloneq\left(\{a,z\},\,\overline{az} \right)$.
We can iteratively obtain
 $D_n$ from $D_{n-1}$ by replacing each edge of $D_{n-1}$ with $b$ parallel {branches consisting of} $s$ consecutive edges (hence we also add $b(s-1)$ new vertices {for every edge removed}). See Figure~\ref{fig:MKlattice} for an example.

We start by pointing out some properties of these graphs. The fact that $\abs{E_n}=(bs)^n$ is immediate from the definition of the model. For each edge at level $n-1$ there are $b(s-1)$ new vertices {at level $n$; therefore} $$\abs{V_n}-\abs{V_{n-1}}=(bs)^{n-1}b(s-1).$$ {Hence, for all $n\in \mathbb{N}$,}
\begin{equation}\label{eq:V and E}
    \abs{V_n}=2+\dfrac{(s-1)}{s}\sum_{i=1}^n(bs)^i=(bs)^n\dfrac{b(s-1)}{bs-1}+\dfrac{b-1}{bs-1}+1. 
\end{equation}
{Also note that $\deg_n\left(a\right)=\deg_n\left(z\right)=b^n$.}

We also introduce {the following shorthand}, which represents the number of edges of a spanning tree of $D_n$: 
\begin{equation}\label{def:c_n}
    c_n:=\abs{V_n}-1.
\end{equation}

We now introduce a special class of vertices in $V_n$, which are the ones for which we prove~Theorem~\ref{thm:powerlaw}. Fix $n \in \mathbb{N}_+$. A collection of \emph{pivotal points} $(p_j)_{j\in \{1,\dots,n-1\}}\subset V_n$ in $D_n$ satisfies
\begin{subequations}\label{def:pivotal}
\begin{align}     d_n\left(a,\,p_j\right)&=s^j,&& \forall j\in \{1,\,\dots,\,n-1\},\label{eq:eq_a}\\
d_n(p_j,\,p_{j+1})&=s^j(s-1),&& \quad \forall j\in \{1,\,\dots,\,n-2\}.
\end{align}
\end{subequations}
There are several possible choices of $p_j$'s satisfying~\eqref{eq:eq_a}, but once a branch is chosen (for example, the uppermost branch in the orientation of the graph given by {Figure~\ref{fig:MKlattice}), we uniquely} determine the sequence $p_k$ for $k>j$. Our results hold regardless of the choice of the branch. In this way the distance between $a$ and $p_j$ remains constant in each $D_n$. Notice that $\deg_n(p_j)=2b^j$ for all $n>j$.

For the boundary condition used in Sections~\ref{sec:UB_uniform}--\ref{sec:R'}, recall the notation
\[
 \partial=\{a,z\},\qquad U_n:=V_n\setminus\partial,\qquad
 G_n=D_n/\partial,\qquad G_n^x=G_n/\{\partial,x\}\quad(x\in U_n).
\]
Thus $V(G_n)=U_n\cup\{\partial\}$, with $\partial$ now a single sink vertex, and $G_n^x$ is equivalently obtained by wiring $\{a,z,x\}$ in $D_n$. We use
\[
 \nu_n:=\nu_{G_n},\qquad
 \mathbb{E}_n:=\mathbb{E}_{\nu_n},\qquad
 \mathbb{P}_n^x:=\mathbb{P}_{G_n^x}.
\]
Under $\mathbb{P}_n^x$, $T^x$ is the $x$-component of the two-root forest on $G_n$ obtained by separating $x$ from $\partial$. Its nonsink vertices are identified with vertices of $U_n$. Throughout, $d_n$ and $\deg_n$ are computed in $D_n$; in particular, wiring the boundary does not change the degree of a vertex in $U_n$. We write $\mathcal{P}_n^v$ and $\mathcal{E}_n^v$ for the law and expectation of the random walk on $D_n$ started at $v$, and abbreviate its boundary hitting time to $\tau_\partial$. This keeps the two endpoints available in the geometric and random-walk arguments.

\subsection{The random walk Green's function}\label{sec:green}

{In this section, we} present some results about the Green's function for the simple symmetric random walk on $D_n$, {which will be used below for the ASM on $G_n$. We first recall the corresponding facts for a finite, connected, undirected multigraph $G=(V\cup\{\partial\},E)$ with sink $\partial$.} Define the Green's function for the simple symmetric random walk on $G$ killed when it reaches ${\partial}$ as
\begin{equation*}
    \mathcal{G}_G\left(x,y\right):=\mathbb{E}^x\left[\sum_{i=0}^{\tau_{{\partial}}-1}\ONE\left\{X_i=y\right \}\right], \quad \forall x,y \in V\cup \{{\partial}\}.
\end{equation*}
The connection between the ASM on $G$ and $\mathcal{G}_G$ is shown in~\citet[Exercise $3.4$]{JaraiSurvey}, {through the following equation involving} the inverse matrix of $\Delta$ as defined in~\eqref{eq:delta}:
\begin{equation}\label{eq:Delta-G}
\mathcal{G}_G(x,y)=\deg_G(y)\left(\Delta^{-1}\right)(x,y), \quad \forall x,y\in V. 
\end{equation}
{We have already shown} in~\eqref{eq:E[N]} and~\eqref{eq:density=E[N]} how the matrix $\Delta^{-1}$ contains information about the sandpile dynamics. While one could use some linear algebra together with the recursive structure of our graph to compute specific entries of $\Delta^{-1}$, we prefer to compute $\mathcal{G}_G$ with probabilistic arguments.

As the reader will recall, the reversibility of the random walk yields many useful results, {one of which is given in}~\citet[Lemma 4.6.1]{lawler}:
\begin{equation}\label{lm:Green}
\mathcal{P}^x_G\left(\tau_y<\tau_{{\partial}}\right)=\dfrac{\mathcal{G}_G\left(x,y\right)}{\mathcal{G}_G\left(y\right)},\, \forall x,\,y\in V.
\end{equation}
Another useful equality is shown in \citet[Exercise 2.1]{lyons2017probability}, {using the fact that the invariant distribution} of the walk is proportional to the degree: 
\begin{equation}\label{eq:time rev}
\mathcal{G}_G(x,y)=\dfrac{\deg_G(y)}{\deg_G(x)}\mathcal{G}_G(y,x),\, \forall x,\,y\in V.
\end{equation}
The symmetry of $D_n$ will allow us to rephrase many questions about random walks on the graph in terms of a one-dimensional random walk on a subset of $\mathbb{Z}$. Hence, for the reader's convenience, we recall two classical results on the gambler's ruin problem. Consider a simple symmetric random walk on $\{a,\,\dots,\,z\}$ with $a,z \in \mathbb{Z},\,z>a$, and $X_0=x\in \{a,...,z\}$. {Then, for $x,y\in \{a,\,\dots ,\,z\}$,}
\begin{equation}
    \label{eq:gambler}    \mathcal{P}^x_{\{a,\dots, z\}}\left(\tau_z<\tau_a\right)=\dfrac{x-a}{z-a},\quad \mathcal{E}^x_{\{a,\dots,z\}}\left[\sum_{i=0}^{\tau_{\{a,z\}}-1}\ONE\left\{X_i=y\right\}\right]=\dfrac{2(x\wedge y-a)(z-x\vee y)}{z-a}.
\end{equation}
{For references, see}~\citet[Theorem 4.8.7]{durrett} and~\citet[Theorem 4.6.1]{lawler}.

\subsubsection{Green's function on the hierarchical lattice}\label{subsec:Green}
Fix $n\geq2$ and pivotal points $(p_k)_{1\leq k<n}$ as in~\eqref{def:pivotal}. The walks on $D_n$ and $G_n$ have the same law before absorption at $\partial$. We use $\tau_\partial$ for this absorption time in either graph: it means hitting $\{a,z\}$ in $D_n$ and hitting the wired sink in $G_n$. In particular, for $x,y\in U_n$,
\begin{equation*}
 \mathcal{G}_n(x,y):=
 \mathcal{E}_{D_n}^{x}\!\left[\sum_{i=0}^{\tau_\partial-1}\ONE\{X_i=y\}\right]
 =\mathcal{E}_{G_n}^{x}\!\left[\sum_{i=0}^{\tau_\partial-1}\ONE\{X_i=y\}\right]
 =\mathcal{G}_{G_n}(x,y).
\end{equation*}
We extend $\mathcal{G}_n$ by zero when either argument belongs to $\partial$. The identity~\eqref{eq:Delta-G} applies with $G=G_n$. We use the unwired graph $D_n$ for the computations below, so that events involving $\tau_a$ and $\tau_z$ retain their meaning.

The main result of this section is the computation of the absorption probability $\mathcal{P}^{p_k}_n\left(\tau_{p_j}<{\tau_\partial}\right)$, for $j<k<n$, which is achieved in Corollary~\ref{cor:absorb}. Its importance is related to Proposition~\ref{prop:GK}.
We start with a simpler case in which one can exploit some symmetry to reduce the problem to the gambler's ruin.
\begin{lemma}\label{lm:gambler G}
    It holds that
    \begin{equation*}
        \mathcal{G}_n(p_{n-1})=s^{n-1}\dfrac{2(s-1)}{s}.
    \end{equation*}
\end{lemma}
\begin{figure}
\begin{tikzpicture}[x=1cm,y=1cm,line width=.85pt,
 line cap=round,line join=round,every node/.style={font=\Large},
 upper/.style={dash pattern=on 3pt off 2pt},
 middle/.style={pattern=horizontal lines,pattern color=black!70},
 lower/.style={pattern=dots,pattern color=black}]
\coordinate (A) at (0,0);
\coordinate (B) at (12,0);
\foreach \i/\y/\sty in {1/2.35/upper,2/0/middle,3/-2.35/lower}{
  \coordinate (p\i3) at (3,\y);
  \coordinate (p\i2) at (6,\y);
  \coordinate (p\i1) at (9,\y);
  \foreach \a/\b in {3/2,2/1}{
    \draw[\sty] (p\i\a)
      .. controls +(0.9,.68) and +(-0.9,.68) .. (p\i\b)
      .. controls +(-0.9,-.68) and +(0.9,-.68) .. (p\i\a) -- cycle;
  }
}
% Original rounded end bubbles.
\draw[upper] (A) .. controls (-.2,1.5) and (1.5,2.35) .. (p13)
 .. controls (2.6,1.15) and (1.2,1.0) .. (A) -- cycle;
\draw[upper] (p11) .. controls (10.5,2.35) and (12.2,1.5) .. (B)
 .. controls (10.8,1.0) and (9.4,1.15) .. (p11) -- cycle;
\draw[middle] (A) .. controls (1,.45) and (2,.45) .. (p23)
 .. controls (2,-.45) and (1,-.45) .. (A) -- cycle;
\draw[middle] (p21) .. controls (10,.45) and (11,.45) .. (B)
 .. controls (11,-.45) and (10,-.45) .. (p21) -- cycle;
\draw[lower] (A) .. controls (-.2,-1.5) and (1.5,-2.35) .. (p33)
 .. controls (2.6,-1.15) and (1.2,-1.0) .. (A) -- cycle;
\draw[lower] (p31) .. controls (10.5,-2.35) and (12.2,-1.5) .. (B)
 -- (p31) -- cycle;
\foreach \i in {1,2,3}{
  \foreach \j in {1,2,3}{\fill (p\i\j) circle[radius=1.65pt];}
}
\fill (A) circle[radius=1.65pt];
\fill (B) circle[radius=1.65pt];
\node[left=10pt] at (A) {$a$};
\node[right=10pt] at (B) {$z$};
\end{tikzpicture}
\caption{In $D_n^{3,4}$ there are three chains, represented as the dashed, hatched and dotted subgraphs respectively.}
\label{fig:chain}
\end{figure}
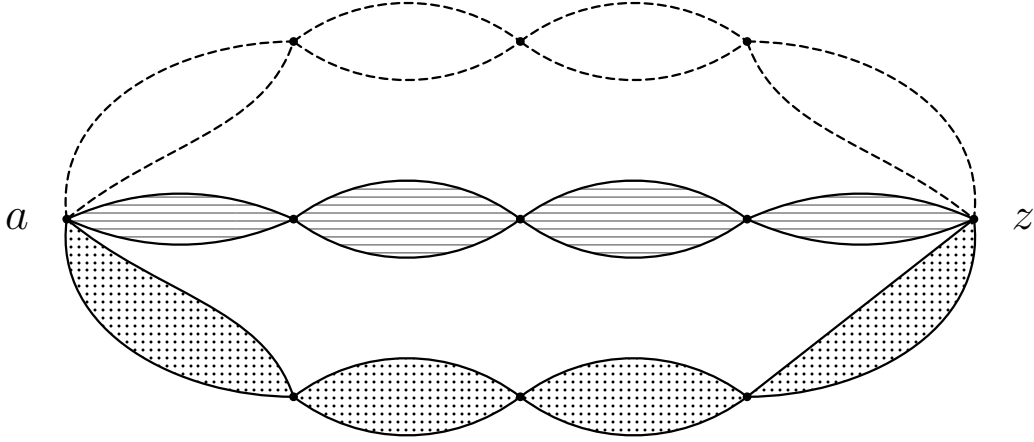
\begin{proof}    
    Take $X_0=p_{n-1}$. Define a {\em chain} as one of the $b$ subgraphs of $D_n$ that {consist of $s$ copies} of $D_{n-1}$ glued together between $a$ and $z$ (see Figure~\ref{fig:chain}). The path $\left(X_k\right)_{k\leq {\tau_\partial}}$ is contained in the chain where $p_{n-1}$ lies. Since in this chain $p_{n-1}$ is the unique vertex such that $d_n(a,\,p_{n-1})=s^{n-1}$, we can look at the projected random walk $\widetilde{X}_\cdot:=d_n\left(a,X_\cdot\right)$. In fact, it holds that $X_k=p_{n-1}$ if and only if $\widetilde{X}_k=s^{n-1}$ and ${\tau_\partial}=\widetilde{\tau}_{\{0,s^n\}}$. In addition, the symmetry of the graph yields that $\widetilde{X}$ is distributed as a symmetric random walk over the set $\left\{0,\,\dots,\,s^n\right\}$, so that we can use the result from the gambler's ruin problem~\eqref{eq:gambler}:
    \begin{equation*}
        \mathcal{G}_n(p_{n-1})=\mathcal{E}_{\{0,\dots,s^n\}}^{s^{n-1}}\left[\sum_{i=0}^{\tau_{\{0,s^{n}\}}}\ONE\{X_i=s^{n-1}\}\right]=s^{n-1}\dfrac{2(s-1)}{s},
    \end{equation*}
    which is exactly the expression in the claim.
\end{proof}
{Having an expression for $\mathcal{G}_{n}(p_{n-1})$, we} can exploit the hierarchical structure of the graph to obtain $\mathcal{G}_n(p_j,\,p_k)$ by expressing it in terms of $\mathcal{G}_n(p_{n-1})$. {We use the notation $\mathcal{E}_n^x=\mathcal{E}_{D_n}^x$ fixed above.} 
\begin{prop}\label{prop:G_pivotal}
For all $j,\,k\in \{1,\,\dots,\,n-1\}$, it holds that    \begin{equation*}
       \mathcal{G}_{n}(p_j,p_k)=
       \begin{cases}
       \displaystyle{}
        s^j\dfrac{2b(s-1)}{(bs-1)}(1-(bs)^{k-n}),\,& j\leq k,\\
         b^{k-j}s^k\dfrac{2b(s-1)}{bs-1}(1-(bs)^{j-n})  ,\,&j>k.  
         \end{cases}
\end{equation*}
   {Therefore, for any fixed} $j$,
\begin{equation*}
\sup_{n>j}\sup_{k\in \{1,\dots,n-1\}}\mathcal{G}_n(p_j,p_k)<\infty.
\end{equation*}
 \end{prop} 
\begin{proof}
    We start by computing $\mathcal{G}_n(p_j)$. {Introduce the shorthand} $E_{i,j}:=\left\{X_i=p_j\right\}$. Then we split the returns to $p_j$ before and after (possibly) touching $p_n$ as follows:
\begin{align}       
        \mathcal{G}_{n+1}(p_j)&=\clE_{n+1}^{p_j}\left[\ONE\{\tau_{p_n}<\tau_a\} \sum_{i=0}^{{\tau_\partial}}\ONE\{E_{i,j}\}\right]+\clE_{n+1}^{p_j}\left[\ONE\{\tau_{p_n}>\tau_a\} \sum_{i=0}^{{\tau_\partial}}\ONE\{E_{i,j}\}\right]\nonumber\\
        &=\clE_{n+1}^{p_j}\left[\ONE\{\tau_{p_n}<\tau_a\}\left( \sum_{i=0}^{\tau_{p_{n}}-1}\ONE\{E_{i,j}\}+\sum_{i=\tau_{p_{n}}}^{{\tau_\partial}}\ONE\{E_{i,j}\}\right)\right]\nonumber\\
        &\quad+\clE_n^{p_j}\left[\ONE\{\tau_z>\tau_a\} \sum_{i=0}^{{\tau_\partial}}\ONE\{E_{i,j}\}\right]\label{eq:exp_G_par}.
    \end{align}In the second equality, {knowing that on $\{\tau_{p_n}>\tau_a\}$,} the random walk in the second expectation {has the same law as a random walk} on $D_n$ in which $p_n$ is mapped into $z$.

Let us analyse the first summand of~\eqref{eq:exp_G_par}. We use the Markov property to conclude that
\begin{align*}   \clE_{n+1}^{p_j}\left[\ONE\{\tau_{p_n}<\tau_a\}\sum_{i=\tau_{p_n}}^{{\tau_\partial}}\ONE\{E_{i,j}\}\right]&=\mathcal{P}_{n+1}^{p_j}\left(\tau_{p_n}<\tau_a\right)\clE^{p_n}_{n+1}\left[\sum_{i=0}^{{\tau_\partial}}\ONE\{E_{i,j}\}\right]\\
&=\mathcal{P}_{n+1}^{p_j}\left(\tau_{p_{n}}<\tau_a\right)\mathcal{G}_{n+1}\left(p_n,p_j\right).
\end{align*}
{Secondly, conditioning on $\{\tau_{p_n}<\tau_a\}$,} we deduce that
\begin{equation*}    \clE_{n+1}^{p_j}\left[\ONE\{\tau_{p_n}<\tau_a\}\sum_{i=0}^{\tau_{p_{n}}-1}\ONE\{E_{i,j}\}\right]=\clE_n^{p_j}\left[\ONE\{\tau_z<\tau_a\} \sum_{i=0}^{{\tau_\partial}}\ONE\{E_{i,j}\}\right].
\end{equation*}
Putting the above two expressions into~{\eqref{eq:exp_G_par}, one obtains}
\begin{equation}\label{eq:utility}
    \mathcal{G}_{n+1}(p_j)=\mathcal{G}_n(p_j)+\mathcal{P}_{n+1}^{p_j}\left(\tau_{p_{n}}<\tau_a\right)\mathcal{G}_{n+1}\left(p_n,p_j\right).
\end{equation}
Using first the reversibility of the Green's function~\eqref{eq:time rev} and then~{\eqref{lm:Green}, one has} that

\begin{equation*}
    \mathcal{G}_{n+1}\left(p_n,p_j\right)=\dfrac{\deg(p_j)}{\deg(p_n)}\mathcal{G}_{n+1}(p_j,p_n)=b^{j-n}\mathcal{P}_{n+1}^{p_j}\left(\tau_{p_n}<{\tau_\partial}\right)\mathcal{G}_{n+1}(p_n).
\end{equation*}
{Substituting the above expression into~}\eqref{eq:utility} and using that $\mathcal{P}_{n+1}^{p_j}\left(\tau_{p_n}<{\tau_\partial}\right)=\mathcal{P}_n^{p_j}\left(\tau_z<\tau_a\right)$, one has that
\begin{equation*}
    \mathcal{G}_{n+1}(p_j)=\mathcal{G}_n(p_j)+b^{j-n}\mathcal{G}_{n+1}(p_n)\left(\mathcal{P}_n^{p_j}\left(\tau_z<\tau_a\right)\right)^2.
\end{equation*}
Using the projected random walk $\widetilde{X}$ as we did in Lemma~\ref{lm:gambler G}, one can use the gambler's ruin result~\eqref{eq:gambler} to prove that
\begin{equation*}    \mathcal{P}^{p_j}_n\left(\tau_z<\tau_a\right)=\dfrac{d_n\left(a,p_j\right)}{d_n(a,z)}=s^{j-n},
\end{equation*}
from which we can obtain, using Lemma~\ref{lm:gambler G} twice, that
\begin{equation*}
    \mathcal{G}_{n+1}(p_j)-\mathcal{G}_n(p_j)=(bs)^{j-n}\mathcal{G}_{j+1}(p_j).
\end{equation*}
Now we write $\mathcal{G}_n(p_j)$ as a telescoping sum to exploit the expression for the increment in $n$ just obtained:
\begin{equation*}
    \mathcal{G}_n(p_j)=\mathcal{G}_{j+1}(p_j)+\sum_{i=j+1}^{n-1}\left[\mathcal{G}_{i+1}(p_j)-\mathcal{G}_i(p_j)\right]=\mathcal{G}_{j+1}(p_j)+\mathcal{G}_{j+1}(p_j)\sum_{i=j+1}^{n-1}(bs)^{j-i},
\end{equation*}
from which the expression for $\mathcal{G}_{j+1}(p_j)$ obtained in Lemma~\ref{lm:gambler G} yields that
\begin{equation}\label{eq:only_now}
    \mathcal{G}_n(p_j)=s^j\dfrac{2b(s-1)}{bs-1}(1-(bs)^{j-n}).
\end{equation}
It remains to compute $\mathcal{G}_n(p_j,p_k)$ for $ k\in\{1,\dots,n-1\},\,k\neq j$. {For $j<k$, use}~\eqref{lm:Green} to get
\begin{equation*}
    \mathcal{G}_n(p_j,p_k)=\mathcal{P}^{p_j}_n\left(\tau_{p_k}<{\tau_\partial}\right)\mathcal{G}_n(p_k).
\end{equation*}
Since $X_{\cdot}$ cannot reach $z$ without passing through $a$ or $p_k$, we compute $\mathcal{P}^{p_j}_n\left(\tau_{p_k}<\tau_a \right)$. Using the projected random walk as done in {Lemma~\ref{lm:gambler G}, it can} be shown that it is equal to $s^{j-k}$. {Hence, by~}\eqref{eq:only_now}
\begin{equation*}
    \mathcal{G}_n(p_j,p_k)=s^j\dfrac{2b(s-1)}{bs-1}(1-(bs)^{n-k}),\quad \forall\,  k>j.\qedhere
\end{equation*}
\end{proof}
Having the values of the Green's function at the pivotal points, we can compute {the absorption probability at $p_j$} for a random walk started at $p_k$ and killed at $\{a,\,z\}$ for $k>j$.
\begin{corollary}\label{cor:absorb}
{For all $j\leq k <n$, it holds} that
\begin{equation*}
    \mathcal{P}^{p_k}_n\left(\tau_{p_j}<{\tau_\partial}\right)=b^{j-k}\dfrac{1-(bs)^{k-n}}{1-(bs)^{j-n}}.
\end{equation*}   
\end{corollary}
In fact, such a result is a direct consequence of Proposition~\ref{prop:G_pivotal} and~\eqref{lm:Green}.

As a consequence of Proposition~\ref{prop:G_pivotal} and Dhar's formula~{\eqref{eq:E[N]}, we also} obtain an expression for {the average} of $N(\cdot,\,\cdot)$.
\begin{corollary}\label{cor:E[N]}{Under $\nu_n$, the expected number of topplings} at $p_k$ in an avalanche {with source $p_j$} is
\begin{equation*}
    \mathbb{E}_n\left[N(p_j,p_k)\right]=\dfrac{s^{j\wedge k}}{b^{j\vee k}}\,\dfrac{b(s-1)}{bs-1}\,\left(1-(bs)^{(j\vee k)-n}\right).
\end{equation*}
\end{corollary}

    As shown in~\eqref{eq:density=E[N]}, {the above corollary} yields the {normalisation} factor {when passing from $\mathrm{R}_{G_n}$ to $\mathrm{R}_{G_n^{p_j}}$.} Note that such a factor is uniformly bounded in $n$.

\section{Upper bound for the uniform wave (Theorem~\ref{thm:powerlaw_wave})}\label{sec:UB_uniform} {In this section, we} prove an upper bound on the {power-law} decay for the size of a uniformly chosen wave. By Remark~\ref{rem:link_with_uni_wave} and bijection~\eqref{bij IKP wave}, this is {equivalent to proving} the result for {$\abs{V(T^{p_j})}$ under $\mathbb{P}_n^{p_j}$}. {Moreover, by Lemma}~\ref{lm:UB by tree}, we obtain {an upper bound on the tail of the size} of any wave cluster.

{For the rest of this section, we fix} $n\in \{2,\,3,\,\ldots\}$, $j\in \{1,\dots,n-1\}$ and a choice of $p_j\in V_n$. {The sandpile measure is $\nu_n$ on $G_n$, while the uniform-wave forest at $p_j$ has law $\mathbb{P}_n^{p_j}$, represented by spanning trees of $G_n^{p_j}$.} 
\begin{proof}[Proof of the upper bound of Theorem~\ref{thm:powerlaw_wave}]
    {To obtain an upper bound on} ${\mathbb{P}_n^{p_j}}\left(\abs{V\left(T^{p_j}\right)}>t\right)$ we start by picking specific values of $t$ for which we can use the pivotal points. Define the sequence $(t_k)_{k\in \{j,\,\dots,\,n-2\}}$ as
    \begin{equation}\label{def:t_k}
    t_k:=s\left(\abs{V_k}-1\right)=\frac{bs(s-1)(bs)^k+s(b-1)}{bs-1},
    \end{equation}
    by~\eqref{eq:V and E}.

    For the first part of the proof, fix also $k\in \{j,\dots,n-2\}$. Looking at the schematic situation in Figure~\ref{fig:UB}, the vertices between $a$ and $p_{k+1}$ (shaded area) in the chain containing $p_j$ {number $t_k$} (there are $s$ copies of $D_k$ between $a$ and $p_{k+1}$). Therefore $T^{p_j}$ has more than $t_k$ vertices only if $p_{k+1}\in T^{p_j}$, proving that
    \begin{equation*}        {\mathbb{P}_n^{p_j}}\left(\abs{V(T^{p_j})}> t_k\right)\leq {\mathbb{P}_n^{p_j}}\left(p_{k+1}\in T^{p_j}\right).
    \end{equation*}
    
    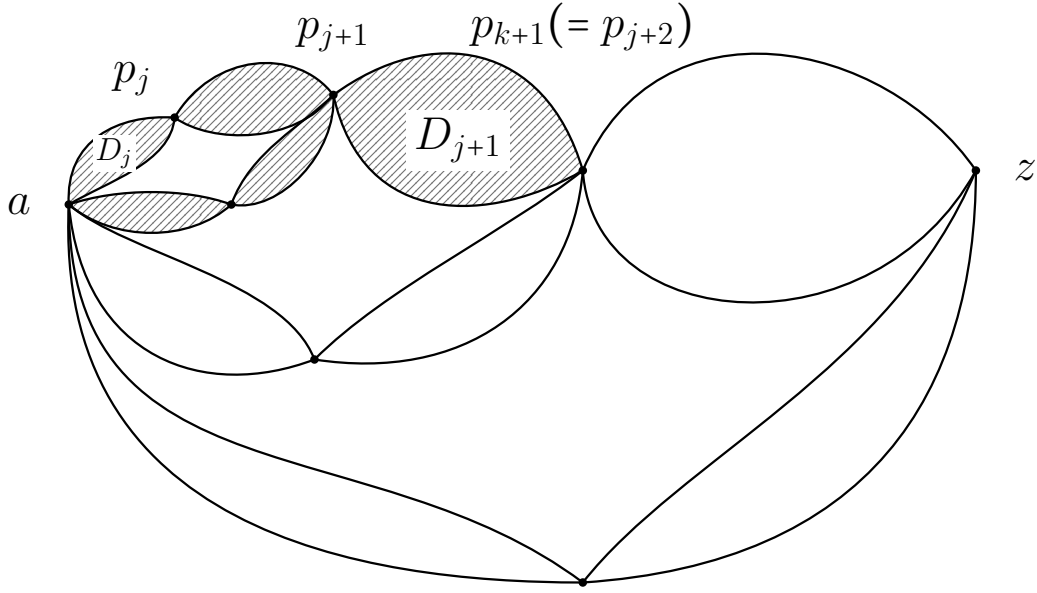
\begin{figure}[htbp]
  \begin{tikzpicture}[line width=.85pt,line cap=round,line join=round,
 every node/.style={font=\Large},
 shaded/.style={pattern=north east lines,pattern color=black!55}]
\coordinate (A) at (0,0);
\coordinate (Pj) at (1.4,1.15);
\coordinate (Pnext) at (3.5,1.45);
\coordinate (Q) at (2.15,0);
\coordinate (R) at (6.8,.45);
\coordinate (S) at (3.25,-2.05);
\coordinate (B) at (12,.45);
\coordinate (T) at (6.8,-5);

% Four small lenses in the left-hand subgraph.
\draw[shaded] (A) .. controls (-.05,.85) and (.65,1.2) .. (Pj)
 .. controls (1.35,.55) and (.5,.3) .. (A) -- cycle;
\draw[shaded] (Pj) .. controls (1.9,2.0) and (3.0,2.1) .. (Pnext)
 .. controls (3.0,.8) and (1.9,.8) .. (Pj) -- cycle;
\draw[shaded] (A) .. controls (.65,.22) and (1.55,.22) .. (Q)
 .. controls (1.6,-.5) and (.6,-.5) .. (A) -- cycle;
\draw[shaded] (Q) .. controls (2.4,.7) and (3.1,1.05) .. (Pnext)
 .. controls (3.55,.7) and (2.8,-.12) .. (Q) -- cycle;

% The next shaded lens.
\draw[shaded] (Pnext) .. controls (4.8,2.4) and (6.3,2.15) .. (R)
 .. controls (5.6,-.25) and (3.9,-.35) .. (Pnext) -- cycle;

% Lower two lenses of the middle subgraph.
\draw (A) .. controls (.75,-.65) and (2.85,-1.0) .. (S);
\draw (A) .. controls (.2,-2.05) and (1.8,-2.6) .. (S);
\draw (S) .. controls (4.2,-1.1) and (5.8,-.35) .. (R);
\draw (S) .. controls (5.2,-2.35) and (6.65,-1.35) .. (R);

% Three large outer lenses.
\draw (R) .. controls (7.8,2.8) and (10.8,2.2) .. (B);
\draw (R) .. controls (6.95,-1.75) and (10.65,-2.0) .. (B);
\draw (A) .. controls (0,-4.0) and (4.0,-2.9) .. (T);
\draw (A) .. controls (-.2,-4.3) and (3.8,-5.0) .. (T);
\draw (T) .. controls (8.1,-3.35) and (10.9,-2.2) .. (B);
\draw (T) .. controls (10.6,-4.75) and (12.0,-2.5) .. (B);

% Marked vertices and labels (the top equation is intentionally omitted).
\foreach \p in {A,Pj,Pnext,Q,R,S,B,T}
  \fill (\p) circle[radius=1.6pt];
\node[left=10pt] at (A) {$a$};
\node[right=10pt] at (B) {$z$};
\node[above left=6pt] at (Pj) {$p_j$};
\node[above=12pt] at (Pnext) {$p_{j+1}$};
\node[above=41pt] at (R) {$p_{k+1}(=p_{j+2})$};
\node[fill=white,inner sep=1pt,font=\normalsize] at (.60,.68) {$D_j$};
\node[fill=white,inner sep=3pt] at (5.15,.85) {$D_{j+1}$};
\end{tikzpicture}
    \caption{Sketch for the proof of the upper bound of Theorem~\ref{thm:powerlaw_wave}.}
    \label{fig:UB}
    \end{figure}   

    {By Proposition~\ref{prop:GK}, the above} {right-hand} side can be translated {in terms of a random walk as} ${\mathcal{P}^{p_{k+1}}_n}\left(\tau_{p_j}<{\tau_\partial}\right)$. Plugging in the value computed in {Corollary~\ref{cor:absorb}, we obtain} the following upper bound:
    \begin{equation}\label{eq:bound_t_k}
        {\mathbb{P}_n^{p_j}}\left(\abs{V\left(T^{p_j}\right)}>t_k\right)\leq b^{j-(k+1)}\dfrac{1-(bs)^{k+1-n}}{1-(bs)^{j-n}}<b^{j-(k+1)}. 
    \end{equation}
    Note that the above upper bound is uniform in $n$.

    The second part of the proof extends the tail bound outside the range of {the $t_k$ values; therefore, fix} $t$ such that $t_j\leq t<s\left(\abs{V_n}-1\right)$.  We need to find a function $f:t\mapsto f(t)$ such that
    \begin{equation}
      \label{eq:argmax}  
    \lfloor f(t) \rfloor=\argmax_{\ell\in\{j,\dots,n-1\}}\{t_\ell\leq t\}.
    \end{equation}
    {For all $\ell\in \{j,\,\dots,\,n-1 \}$, we have}
    \begin{equation*}
        t\geq s\left(\abs{V_\ell}-1\right)\iff \ell\leq \log_{bs}\left(\dfrac{bs-1}{bs(b-1)}t-\dfrac{b-1}{b(s-1)}\right),
    \end{equation*}
    {so}
        \begin{align}{\label{def:f}}
       f(t):=&\log_{bs}\left(K_1t-K_2\right)     
    \end{align} 
   achieves~\eqref{eq:argmax}, having defined    
    \begin{align*}
        K_1:=\dfrac{bs-1}{bs(s-1)}&, \, K_2:=\dfrac{b-1}{b(s-1)}.
    \end{align*}        
    Now one can see the {power-law} decay with exponent $\alpha=1/\log_b(bs)$ as stated in Theorem~\ref{thm:powerlaw_wave}. Use~\eqref{eq:bound_t_k} to get that
    \begin{align}\label{eq:extension}
        {\mathbb{P}_n^{p_j}}\left(\abs{V\left(T^{p_j}\right)}>t\right)&<b^{j-\left(\floor{f(t)}+1\right)}\leq b^j\,b^{-f(t)}=b^j\, \left(K_1t-K_2\right)^{-\alpha}\nonumber\\
        &\leq b^j\,t^{-\alpha}\left(\dfrac{K_1t_j-K_2}{t_j}\right)^{-\alpha}=b^j\,t^{-\alpha}\, \left(\dfrac{(bs)^j}{t_j}\right)^{-\alpha},
    \end{align}
    where we used that $\left(K_1t-K_2\right)/t$ is an increasing function for positive $t$ and that $K_1t_j-K_2=(bs)^j$ by the definition of $f$.

    {Hence, introducing the constant}
    \begin{equation*}
        {K}_u:=b^j\left(\dfrac{bs(s-1)}{bs-1}+\dfrac{s(b-1)}{(bs)^j(bs-1)}\right)^{\alpha},
    \end{equation*}
    {derived from~\eqref{def:t_k}, we have} proved an upper bound for the tail for $t_j\leq t<s\left(\abs{V_n}-1\right)$ as follows:
    \begin{equation}\label{eq:uw_ub}
        {\mathbb{P}_n^{p_j}}\left(\abs{V\left(T^{p_j}\right)}>t\right)<{K}_u\,t^{-\alpha}.
    \end{equation}
\end{proof}
As stated in {the introduction to the current section}, {the upper bound on the tail} of the size of any wave can be obtained as a corollary of Lemma~\ref{lm:UB by tree} together with~\eqref{eq:uw_ub}:
\begin{corollary}\label{cor:w_ub}
For any $i\in \mathbb{N}_+$ and $t\in \left[s\left(\abs{V_j}-1\right),\,s\left(\abs{V_n}-1\right)\right)$ it holds that
    \begin{equation*}
        {\nu_n}\left(\abs{W^{p_j}_i}>t\right)<\left(\dfrac{s}{b}\right)^{j} {K}_u\,t^{-\alpha},
    \end{equation*}
    with the convention $W^{p_j}_i=\emptyset$ if $i>N(p_j)$.
\end{corollary}
\begin{proof}
     Fix $t\in \left[s\left(\abs{V_j}-1\right),\,s\left(\abs{V_n}-1\right)\right)$ and $i\in \mathbb{N}_+$.
     Using Remark~\ref{rem:larger_t} and~{\eqref{eq:E[N]}, we get}
    \begin{equation*}
        {\nu_n}\left(\abs{W^{p_j}_i}>t\right)\leq \mathbb{E}_{{\nu_n}}\left[N(p_j)\right]{\mathbb{P}_n^{p_j}}\left(\abs{V\left(T^{p_j}\right)}>t\right).
    \end{equation*}
    Using the expression provided by {Corollary~\ref{cor:E[N]}, we can} show that $\mathbb{E}_{{\nu_n}}\left[N(p_j)\right]$ {is bounded above by} $(s/b)^j$. {Hence, combining this fact with the upper bound for the uniform wave}~\eqref{eq:uw_ub} yields the expression in the claim: 
    \begin{equation*}
        {\nu_n}\left(\abs{W^{p_j}_i}>t\right)<\left(\dfrac{s}{b}\right)^{j} {K}_u\,t^{-\alpha}.\qedhere
    \end{equation*}
\end{proof}

\section{Lower bound for the uniform wave (Theorem~\ref{thm:powerlaw_wave})}\label{sec:LB_uniform}
{In this section, we} prove the lower bound for the {power-law} decay of the uniform wave. {This, in particular, will imply} the lower bound for the first-wave, which is enough to prove the analogous result for the avalanche.

{We retain the setting of Section~\ref{sec:UB_uniform}: the source is $p_j\in U_n$, the sandpile measure is $\nu_n$, and the uniform-wave forest has law $\mathbb{P}_n^{p_j}$.}

\begin{proof}[Proof of the lower bound of Theorem~\ref{thm:powerlaw_wave}]
Recall the collection $t_k=s\left(\abs{V_k}-1\right) $ for $k\in \{j,\dots,n-3\}$ used to prove the upper bound. We will proceed, as in that proof, by first showing the {power-law} behaviour for all $t_k$'s and then for general admissible values of $t$. Therefore, fix now $t_k$. We argue that
\begin{equation}\label{ineq:LB2}
    {\mathbb{P}_n^{p_j}}\left(\abs{V\left(T^{p_j}\right)}> t_k\right){\geq} {\mathbb{P}_n^{p_j}}\left(p_{k+2}\in T^{p_j}\right).
\end{equation}
{We now prove this claim.} {Each pair of consecutive pivotal points $p_\ell$ and $p_{\ell+1}$ forms the endpoints} of $s-1$ copies of $D_\ell$. {Moreover, if the forest on $G_n$ associated with a tree of $G_n^{p_j}$ realises $\{p_{k+2}\in T^{p_j}\}$, then its path from $p_j$ to $p_{k+2}$} passes through $p_k$ and $p_{k+1}$. Hence all the vertices in the $(s-1)$ copies of $D_{k+1}$ between $p_{k+1}$ and $p_{k+2}$ and the $(s-1)$ copies of $D_k$ between $p_{k}$ and $p_{k+1}$ are in $T^{p_j}$ ({for a graphical representation, see} Figure~\ref{fig:LB}). Therefore
\begin{equation*}
    {\mathbb{P}_n^{p_j}}\left(\abs{V\left(T^{p_j}\right)}\geq (s-1)(\abs{V_{k}}+\abs{V_{k+1}}-2)\right)\geq {\mathbb{P}_n^{p_j}}\left(p_{k+2}\in T^{p_j}\right).
\end{equation*}
Since $(s-1)(\abs{V_k}+\abs{V_{k+1}}-2)>t_k$, this shows the validity of inequality~\eqref{ineq:LB2}.
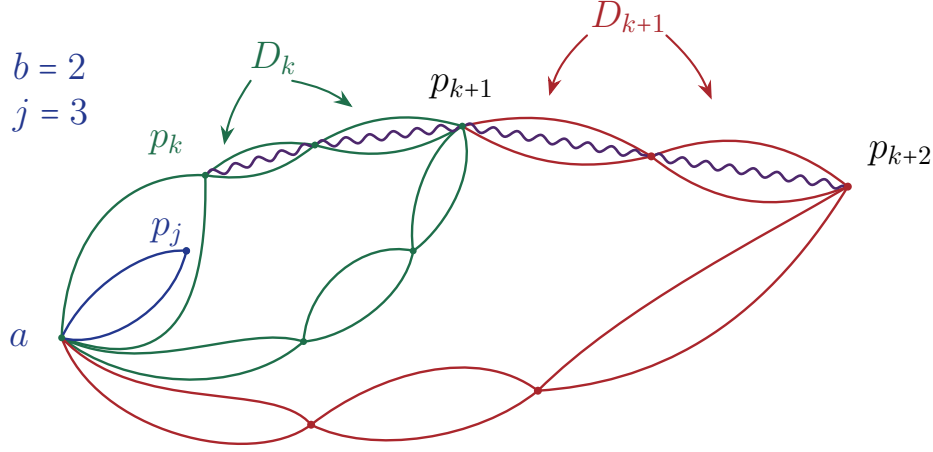
\begin{figure}[htbp]
  \begin{tikzpicture}[line width=.9pt,line cap=round,line join=round,
 every node/.style={font=\large},
 wave/.style={graphpurple,line width=1pt,decorate,
 decoration={snake,amplitude=1.4pt,segment length=9pt}},
 hint/.style={dash pattern=on 4pt off 4pt,line width=.7pt},
 arrow/.style={-{Stealth[length=2.5mm]},line width=.8pt}]
\coordinate (A) at (0,0);
\coordinate (J) at (1.65,1.15);
\coordinate (U) at (1.9,2.15);
\coordinate (V) at (3.35,2.55);
\coordinate (K) at (5.3,2.8);
\coordinate (L) at (3.2,-.05);
\coordinate (M) at (4.65,1.15);
\coordinate (R) at (7.8,2.4);
\coordinate (Z) at (10.4,2.0);
\coordinate (S) at (3.3,-1.15);
\coordinate (T) at (6.3,-.7);

% Green subgraph.
\draw[graphgreen] (A) .. controls (.05,1.4) and (.85,2.2) .. (U)
 .. controls (1.95,-.25) and (1.0,-.35) .. (A);
\draw[graphgreen] (U) to[bend left=24] (V) to[bend left=24] (U);
\draw[graphgreen] (V) to[bend left=23] (K) to[bend left=23] (V);
\draw[graphgreen] (A) .. controls (1.1,-.55) and (2.4,.2) .. (L)
 .. controls (2.35,-.8) and (.8,-.65) .. (A);
\draw[graphgreen] (L) .. controls (3.25,.7) and (4.0,1.3) .. (M)
 .. controls (4.55,.55) and (3.8,.0) .. (L);
\draw[graphgreen] (M) .. controls (4.55,1.85) and (4.95,2.6) .. (K)
 .. controls (5.55,2.1) and (4.95,1.35) .. (M);

% Blue lens ending at p_j inside the left green lens.
\draw[graphblue] (A) .. controls (.3,.65) and (1.1,1.2) .. (J)
 .. controls (1.45,.3) and (.5,-.15) .. (A);

% Red upper chain.
\draw[graphred] (K) to[bend left=22] (R) to[bend left=22] (K);
\draw[graphred] (R) to[bend left=30] (Z) to[bend left=30] (R);

% Red lower chain.
\draw[graphred] (A) .. controls (1.0,-1.05) and (2.8,-.55) .. (S)
 .. controls (2.5,-1.75) and (.45,-1.25) .. (A);
\draw[graphred] (S) .. controls (4.2,-.45) and (5.5,-.1) .. (T)
 .. controls (5.5,-1.35) and (4.1,-1.55) .. (S);
\draw[graphred] (T) .. controls (7.4,.4) and (9.3,1.35) .. (Z)
 .. controls (9.2,.05) and (7.8,-.6) .. (T);

% Wavy highlighted path across the upper subgraphs.
\draw[wave] (U) -- (V);
\draw[wave] (V) -- (K);
\draw[wave] (K) -- (R);
\draw[wave] (R) -- (Z);

% Vertices and retained labels.
\foreach \p in {A,J,U,V,K,L,M}{\fill[graphgreen] (\p) circle[radius=1.3pt];}
\foreach \p in {R,Z,S,T}{\fill[graphred] (\p) circle[radius=1.5pt];}
\fill[graphblue] (J) circle[radius=1.4pt];
\node[left=8pt,graphblue] at (A) {$a$};
\node[graphblue] at (1.40,1.40) {$p_j$};
\node[above=6pt] at (K) {$p_{k+1}$};
\node[above right=5pt] at (Z) {$p_{k+2}$};
\node[graphblue,align=left,anchor=west] at (-.8,3.25) {$b=2$\\$j=3$};
\node[graphgreen] at (2.8,3.7) {$D_k$};
\draw[graphgreen,arrow] (2.5,3.4) to[bend right=12] (2.15,2.55);
\draw[graphgreen,arrow] (3.05,3.4) to[bend left=12] (3.85,3.05);
\node[above left=5pt,graphgreen] at (U) {$p_k$};

\node[graphred] at (7.5,4.25) {$D_{k+1}$};
\draw[graphred,arrow] (6.95,4.0) to[bend right=15] (6.45,3.2);
\draw[graphred,arrow] (8.0,4.0) to[bend left=12] (8.6,3.15);
\end{tikzpicture}
    \caption{Sketch for the proof of the {lower bound in Theorem}~\ref{thm:powerlaw_wave}.}
    \label{fig:LB}
\end{figure}

Proposition~\ref{prop:GK}  and {Corollary~\ref{cor:absorb} yield the expression for} the {right-hand} side of~\eqref{ineq:LB2}, {proving that}
\begin{equation}\label{eq:LB utility}
    {\mathbb{P}_n^{p_j}}\left(\abs{V\left(T^{p_j}\right)}>t_k\right)\geq b^{j-(k+2)}\dfrac{1-(bs)^{k+2-n}}{1-(bs)^{j-n}}\geq b^{j-(k+2)}\dfrac{1-(bs)^{-1}}{1-(bs)^{j-n}}>b^{j-(k+2)}\dfrac{bs-1}{bs}.
\end{equation}

We turn our attention now to a general $t$ such that $t_j<t \leq t_{n-3}$. It is easy to prove that $\lceil f(t)\rceil$, where $f(t)$ is defined in~\eqref{def:f}, is such that
\[
\lceil f(t)\rceil=\argmin_{\ell\in\{j+1,\dots,n-3\}}\{t_{\ell-1}<t\leq t_{\ell}\}.
\]
Thus, applying the lower bound~\eqref{eq:LB utility} for the tail of the wave at $t_{\lceil f(t)\rceil}$, which is at least as big as $t$, we get that
\begin{align*}   {\mathbb{P}_n^{p_j}}\left(\abs{V\left(T^{p_j}\right)}>t\right)&> \dfrac{bs-1}{bs}\,b^{j-(\lceil{f(t)}\rceil +2)}\geq\dfrac{bs-1}{b s}b^{-3}\, b^{-f(t)}\geq \dfrac{bs-1}{b^4 s}\, t^{-\alpha}\left(\dfrac{K_1t_{n-3}-K_2}{t_{n-3}}\right)^{-\alpha}\\
    &=\dfrac{bs-1}{b^4s}\,t^{-\alpha}\left(\dfrac{(bs)^{n-3}}{t_{n-3}}\right)^{-\alpha}=\dfrac{bs-1}{b^4 s}\,t^{-\alpha}\left(\dfrac{bs(s-1)}{bs-1}+(bs)^{-n+3}\dfrac{(b-1)s}{bs-1}\right)^{\alpha}\\
    &\geq \dfrac{bs-1}{b^4 s}\,t^{-\alpha}\left(\dfrac{bs(s-1)}{bs-1}\right)^{\alpha}.
\end{align*}
Again we employed the fact that $(K_1 t-K_2)/t$ is an increasing function, and the definition~\eqref{def:t_k} for $t_{n-3}$.
{Hence, introducing the constant}
\begin{equation*}
    {K}_\ell:=\dfrac{bs-1}{b^4 s}\left(\dfrac{bs(s-1)}{bs-1}\right)^{\alpha},
\end{equation*}
we have proved the desired lower bound for $t_j< t\leq t_{n-3}$:
\begin{equation}\label{eq:uw_lb}
    {\mathbb{P}_n^{p_j}}\left(\abs{V\left(T^{p_j}\right)}>t\right)> {K}_\ell\, t^{-\alpha }.
\end{equation}
\end{proof}
Now, proceeding as in Corollary~\ref{cor:w_ub}, we can use the lower bound for the uniform wave~\eqref{eq:uw_lb} together with the lower bound in~\eqref{eq:lm1} to prove the lower bound for the size of the first-wave.
\begin{corollary}\label{cor:lb_wave}
    For any $t\in \left[s\left(\abs{V_j}-1\right),\,s\left(\abs{V_{n-3}}-1\right)\right]$ it holds that
    \begin{equation*}
        {\nu_n}\left(\abs{W^{p_j}_1}>t\right)>\dfrac{s^{j-1}(s-1)}{Cb^j}\,K_{\ell}\,t^{-\alpha}.
    \end{equation*}
    \begin{proof}
        The lower bound of Remark~\ref{rem:larger_t} {also holds for the event} $\{\abs{W^x_1}>t\}$, which is translated in terms of spanning trees as $\{\abs{V\left(T^x\right)}>t\}$ through the bijection~\eqref{bij IKP wave}. {Applying the comparison lemma with $G=G_n$ and $x=p_j$, we have} the following upper bound on the constant $C$ defined in~\eqref{eq:stima_C}:
        \begin{equation}\label{eq:C_ineq}
            C<\left(2b^{(j-1)}\right)^{s^j-1}.
        \end{equation}
       {Hence, the lower bound} of~\eqref{eq:lm1} yields:
        \begin{equation*} {\nu_n}\left(\abs{W^{p_j}_1}>t\right)> \dfrac{\mathbb{E}_{{\nu_n}}\left[N(p_j)\right]}{C}\,{\mathbb{P}_n^{p_j}}\left(\abs{V\left(T^{p_j}\right)}>t\right).
        \end{equation*}
        Using the expression in Corollary~\ref{cor:E[N]} and the lower bound on the uniform wave~\eqref{eq:uw_lb}, we get
        \begin{equation*}          {\nu_n}\left(\abs{W^{p_j}_1}>t\right)>\left(\dfrac{s}{b}\right)^j\dfrac{1}{C}\left(1-\dfrac{b-1}{bs-1}\right)\left(1-(bs)^{j-n}\right)\,K_\ell\,t^{-\alpha}>\dfrac{s^{j-1}(s-1)}{b^jC}\,K_\ell\,t^{-\alpha},
        \end{equation*}
        proving the {power-law lower bound for the first-wave}.
    \end{proof}
\end{corollary}

\section{{Power-law} for the avalanche (Theorem~\ref{thm:powerlaw})}\label{sec:proofAv}
{In this section, we} prove our main result, Theorem~\ref{thm:powerlaw}. {As in the previous sections}, fix $n\in \{2,\,3,\,\dots\},\, j\in \{1,\,\dots ,\,n-1\}$ {and a choice of $p_j\in V_n$}. 

Of course, the lower bound requires no work at all at this stage, since $W^{p_j}_1\subset Av^{p_j}$ and we have the result for $W^{p_j}_1$ by Corollary~\ref{cor:lb_wave}.
As we are about to see, the upper bound is a consequence of the decay of $\mathbb{E}_{{\nu_n}}[N(p_j,\,p_k)]$ in $k$, as displayed in Corollary~\ref{cor:E[N]}.

\begin{proof}[Proof of Theorem~\ref{thm:powerlaw}, upper bound]
    Let $k\in \{j,\ldots,\,n-2\}$. {As in the proof for the first-wave}, we use the fact that {having more than} $t_{k-1}=s\left(\abs{V_{k-1}}-1\right)$ vertices in $Av^{p_j}$ implies that $p_k\in Av^{p_j}$, hence we need to bound ${\nu_n}\left(p_k\in Av^{p_j}\right)$. We rephrase the event in terms of the toppling number, whose expectation has been computed in Corollary~\ref{cor:E[N]}, and then {use Markov's inequality:}
    \begin{equation}\label{ineq:AvUBpivotal}
        {\nu_n}\left(p_k\in Av^{p_j}\right)={\nu_n}\left(N(p_j,\,p_k)\geq 1\right)\leq \mathbb{E}_{{\nu_n}}\left[N(p_j,\,p_k)\right]<b^{-k} \dfrac{s^jb(s-1)}{bs-1}.
    \end{equation}
   {Hence, for each $k\in \{j,\dots,n-2\}$} we have shown that
   \begin{equation*}
       {\nu_n}\left(\abs{Av^{p_j}}> t_k\right)\leq{\nu_n}\left(p_{k+1}\in Av^{p_j}\right)<b^{-(k+1)} \dfrac{s^jb(s-1)}{bs-1}.
   \end{equation*}
   Now one can {proceed as in the proof} of the upper bound for the size of the uniform wave in Section~\ref{sec:UB_uniform} to extend the result to each $t\in \left[t_j,\, t_n\right)$. Recalling the definition of $f(t)$ given in~{\eqref{def:f}, we finally} obtain that
   \begin{equation*}       {\nu_n}\left(\abs{Av^{p_j}}> t\right)\leq {\nu_n}\left(\abs{Av^{p_j}}> t_{\floor{f(t)}}\right)< b^{-\left(\floor{f(t)}+1\right)}\dfrac{s^jb(s-1)}{bs-1}\leq  \overline{K}\,t^{-\alpha},
   \end{equation*}
   where the constant $\overline{K}$ is obtained by proceeding in a similar way to~\eqref{eq:extension}, which leads to the following definition:
   \begin{equation*}
       \overline{K}:=s^j\dfrac{b(s-1)}{bs-1}\left(\dfrac{bs(s-1)}{bs-1}+\dfrac{s(b-1)}{(bs)^j(bs-1)}\right)^{\alpha}.
   \end{equation*}
   This proves the upper bound shown in Theorem~\ref{thm:powerlaw}.
\end{proof}

\section{Pivotal reach limit (Proposition~\ref{prop:law_R'})}\label{sec:R'}
{Throughout this section, fix $j\in\mathbb{N}_+$. For $n>j$, retain $\nu_n=\nu_{G_n}$ and $\mathbb{P}_n^{p_j}=\mathbb{P}_{G_n^{p_j}}$.}

{In this section, we} measure avalanches by counting how many pivotal points they may contain, and show that this counter is finite almost surely in the limit $n\to\infty$. 

Define the collection of random variables $\left(R_n\right)_{n>j}$ as
\begin{equation}\label{def:R_n'}
    R_n:=\max\left(\left\{k:j\leq k<n,\,p_k\in Av^{p_j}\right\}\cup \{0\}\right).
\end{equation}
We prove Proposition~\ref{prop:R_intro_res} in two steps. Firstly, we show tightness for the collection $\left(R_n\right)_{n>j}$.

\begin{lemma}[Tightness]\label{tightness'}
    The collection of random variables $(R_n)_{n>j}$ is tight. 
\end{lemma}
\begin{proof}
    {Fix $\epsilon>0$; we look} for a positive integer $\ell\geq j$ such that
    \begin{equation*}
        \sup_{n>j}\nu_n\left(R_n\geq \ell\right){\leq} \epsilon.
    \end{equation*}
    The event $\{R_n\geq \ell\}$ happens if and only if $\{p_\ell\in Av^{p_j}\}$. Hence we can use the upper bound proved in~\eqref{ineq:AvUBpivotal}, which happens to be uniform in $n$:
    \begin{equation*}        \nu_n\left(R_n\geq \ell\right)=\nu_n\left(p_\ell\in Av^{p_j}\right)< b^{-\ell} \dfrac{s^jb(s-1)}{bs-1}, 
    \end{equation*}
    proving the tightness condition by taking $\ell$ large enough.
\end{proof}
Since $(R_n)_{n>j}$ is tight, there exists a subsequence $(R_{n_\ell})_{\ell\geq 0}$ that converges in distribution to a random variable that we will denote by $R$. Its law will be denoted by $\nu$. 
\begin{prop}[Law of $R$]\label{prop:law_R'}
    $R$ is {an almost surely finite, non-trivial (i.e. nonconstant) random variable}.
\end{prop}
\begin{proof}
    The upper bound from~\eqref{ineq:AvUBpivotal} also proves that $R$ is {almost surely finite}. In fact, such an inequality is preserved through the limit procedure, yielding that
\begin{equation*}
    \nu(R=+\infty)=\nu\left(\bigcap_{\ell=j}^{\infty}\{R\geq \ell\}\right)=\lim_{\ell\to \infty}\nu\left(R\geq \ell\right)\leq \lim_{\ell\to \infty }b^{-\ell} \dfrac{s^jb(s-1)}{bs-1}=0.
\end{equation*}

Note also that $R$ is a {non-trivial} random variable. In fact, for fixed $n\in \mathbb{N}$ we can obtain a lower bound for its survival function through the first-wave and then use~\eqref{eq:lm1}:
\begin{equation*}
    \nu_n\left(R_n\geq \ell\right)\geq \nu_n\left(p_\ell\in W^{p_j}_1\right)\geq \dfrac{\mathbb{E}_n\left[N(p_j)\right]}{C}{\mathbb{P}_n^{p_j}}\left(p_{\ell}\in V\left(T^{p_j}\right)\right).
\end{equation*}
Using the numerical values computed in~\eqref{eq:C_ineq}, Corollary~\ref{cor:absorb} and {Corollary~\ref{cor:E[N]}, we obtain} the following lower bound:
\begin{equation*}
    \nu_n\left(R_n\geq \ell\right)\geq b^{-\ell}\,\dfrac{s^jb(s-1)}{C\left(bs-1\right)}\,\left(1-(bs)^{\ell-n}\right).
\end{equation*}
Taking the limit in the above formula {(along the appropriate subsequence), we} obtain that 
\begin{equation*}
    \nu\left(R\geq \ell\right)\geq b^{-\ell}\,\dfrac{s^jb(s-1)}{C\left(bs-1\right)},
\end{equation*}
proving that $R$ is almost surely {nonconstant}.
\end{proof}

\subsection{Pivotal reach limit for uniform waves}
{The notation at the start of this section remains the same in this subsection; Section~\ref{sec:Cb} uses the separate endpoint boundary condition.} {In this subsection, we} study the limit of the pivotal reach of a uniform wave. The limiting law can be explicitly computed and is geometric, after shifting. 

{For all $n>j$, work under $\mathbb{P}_n^{p_j}$. Define the pivotal reach of $T^{p_j}$ by}

\begin{equation*}
    L_n:=\max_{j\leq k<n}\left\{p_k\in T^{p_j}\right\}.
\end{equation*}
The careful reader may have noticed a slight {difference from the definition} of $R_n$: this is due to the fact that, unlike avalanches, a wave sampled uniformly contains at least one point, so we can translate our problem directly into spanning forests without considering the empty case.
\begin{lemma}\label{lemma:limit_L}
    For $j\in \mathbb{N}_+$, the random variable $L_n-j$ {converges in law to a $Geom(b^{-1})$ random variable} as $n$ goes to infinity.
\end{lemma}
\begin{proof}
    Corollary~\ref{cor:absorb} yields {an explicit expression for the distribution} of $L_n$. In fact, for all \mbox{$\ell\in \{j,\,\dots \,n-1\}$},
\begin{equation*}
    {\mathbb{P}_n^{p_j}}\left(L_n\geq {\ell}\right)={\mathbb{P}_n^{p_j}}\left(p_{\ell}\in T^{p_j}\right)=b^{j-\ell}\dfrac{1-(bs)^{\ell-n}}{1-(bs)^{j-n}}.
\end{equation*}
To make the connection between $L_n$ and a geometric random variable, consider $L_n-j$, its shifted version. {Then, taking the limit in $n$ of the above formula establishes the existence of a limit} $L$ taking values in $\mathbb{N}$ with survival function
\begin{equation*}
    \mathbb{P}\left(L\geq \ell \right)=b^{-\ell},
\end{equation*}
which coincides with that of a geometric random variable with success probability $b^{-1}$.
\end{proof}
We can provide a heuristic explanation of this fact. In fact, $L_n$ is very close to satisfying the {memoryless property}, which {characterises} geometric random variables among discrete ones. Fix $\ell \in \{1,\,\dots,\,n-2\}$ and {$k\in \{1,\,...,\,n-2-\ell\}$; then}
\begin{equation*}
    {\mathbb{P}_n^{p_j}}\left(L_n-j\geq \ell+k|\,L_n-j\geq \ell\right)={\mathbb{P}_n^{p_j}}\left(p_{j+\ell+k}\in T^{p_j}|\,p_{j+\ell}\in T^{p_j}\right).
\end{equation*}
$L_n-j$ would be memoryless if the above right-hand side were equal to ${\mathbb{P}_n^{p_j}}\left(p_{j+k}\in T^{p_j}\right)$. Although this is not true, using the groundskeeper algorithm one can show that
\begin{equation*}
    {\mathbb{P}_n^{p_j}}\left(p_{j+\ell+k}\in T^{p_j}|\,p_{j+\ell}\in T^{p_j}\right)={\mathbb{P}_n^{p_{j+\ell}}}\left(p_{j+\ell+k}\in T^{p_{j+\ell}}\right).
\end{equation*}
The connection between the two probabilities is given by {Corollary~\ref{cor:absorb}, which states} that ${\mathbb{P}_n^{p_j}}\left(p_{j+k}\in T^{p_j}\right)=b^{-k}(1+o_n(1))={\mathbb{P}_n^{p_{j+\ell}}}\left(p_{j+\ell+k}\in T^{p_{j+\ell}}\right)$.

\section{Size of the uniform wave at \texorpdfstring{$z$}{z} (Theorem~\ref{thm:limit Dbs})}\label{sec:Cb}
In this section the ASM is defined on the original graph $D_n$ with $a$ as its sole sink, and a uniform wave is started at $z$. Its source-wired graph is $D_n/\{a,z\}=G_n$, rather than $G_n^x$ for an additional interior source. Thus the wave--forest correspondence of Section~\ref{subsec:ASM} uses the law $\mathbb{P}_{G_n}$.

Separating $a$ and $z$ in a spanning tree of $G_n$ gives a two-component spanning forest $(T_n^a,T_n^z)$ on $D_n$. We characterise the distribution of
\[
 \mathrm{C}_n:=\abs{V(T_n^z)\setminus\{z\}}
\]
under $\mathbb{P}_{G_n}$. Excluding $z$ gives the convention used in Proposition~\ref{prop:deco Dbs}; adding one recovers the full wave size.

The recursive structure of our graphs can be exploited to express $\mathrm{C}_n$ in terms of $\mathrm{C}_{n-1}$. The key idea is to look at how many \emph{junction vertices} lie in $\mathrm{C}_n$. Junction vertices are {the points connecting the copies of $D_{n-1}$ in the chains} (see Fig.~\ref{fig:decom}).

Fix $n\in \mathbb{N}_+$ and a labelling system for each chain in $D_n$. Denote the junction vertices in $D_n$ as {the double-indexed collection} $(p^{j,k})_{(j,k)\in \{1,...,b\}\times \{1,...,s-1\}}\subset V_n$ satisfying
\begin{equation*}
    d_n\left(z,\,p^{j,k}\right)=ks^{n-1}, \quad p^{j,k}\in \{j\textrm{-th chain}\}.
\end{equation*}
Essentially, the index $j$ encodes the chain {to which $p^{j,k}$ belongs} and $k$ the distance to $z$ ({on the scale $s^{n-1}$}).
Note that each $p^{j,k}$ connects two copies of $D_{n-1}$ so it follows that $\deg_n(p^{j,k})=2b^{n-1}$.

 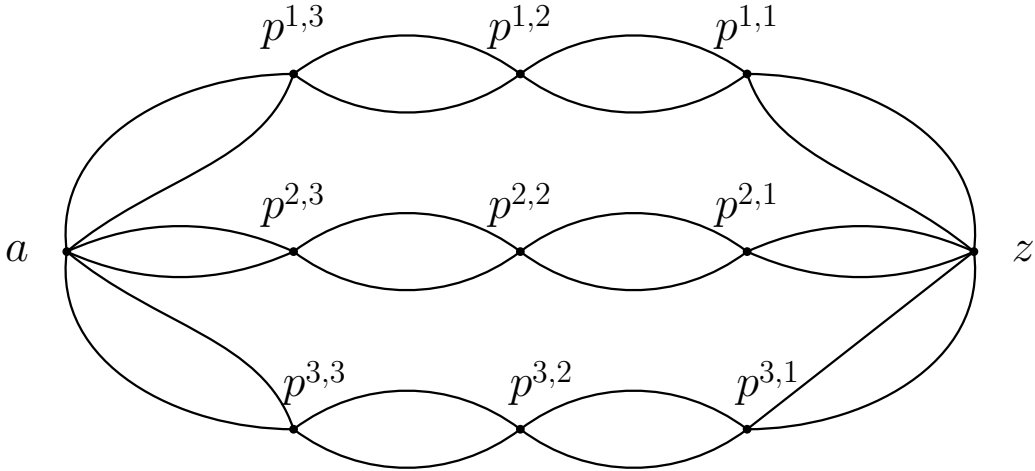
\begin{figure}[ht]  
 \begin{tikzpicture}[x=1cm,y=1cm,line width=0.85pt,
  line cap=round,line join=round,every node/.style={font=\Large}]
  % Three branches; each consecutive pair has two parallel edges.
  \coordinate (A) at (0,0);
  \coordinate (B) at (12,0);
  \foreach \i/\y in {1/2.35,2/0,3/-2.35}{
    \coordinate (p\i3) at (3,\y);
    \coordinate (p\i2) at (6,\y);
    \coordinate (p\i1) at (9,\y);
    \draw (p\i3) .. controls +(0.9,0.68) and +(-0.9,0.68) .. (p\i2);
    \draw (p\i3) .. controls +(0.9,-0.68) and +(-0.9,-0.68) .. (p\i2);
    \draw (p\i2) .. controls +(0.9,0.68) and +(-0.9,0.68) .. (p\i1);
    \draw (p\i2) .. controls +(0.9,-0.68) and +(-0.9,-0.68) .. (p\i1);
  }
  % Upper branch.
  \draw (A) .. controls (-0.2,1.5) and (1.5,2.35) .. (p13);
  \draw (A) .. controls (1.2,1.0) and (2.6,1.15) .. (p13);
  \draw (p11) .. controls (10.5,2.35) and (12.2,1.5) .. (B);
  \draw (p11) .. controls (9.4,1.15) and (10.8,1.0) .. (B);
  % Middle branch.
  \draw (A) .. controls (1,0.45) and (2,0.45) .. (p23);
  \draw (A) .. controls (1,-0.45) and (2,-0.45) .. (p23);
  \draw (p21) .. controls (10,0.45) and (11,0.45) .. (B);
  \draw (p21) .. controls (10,-0.45) and (11,-0.45) .. (B);
  % Lower branch.
  \draw (A) .. controls (-0.2,-1.5) and (1.5,-2.35) .. (p33);
  \draw (A) .. controls (1.2,-1.0) and (2.6,-1.15) .. (p33);
  \draw (p31) .. controls (10.5,-2.35) and (12.2,-1.5) .. (B);
  \draw (p31) -- (B);
  % Vertices and labels.
  \foreach \i in {1,2,3}{
    \foreach \j in {1,2,3}{
      \fill (p\i\j) circle[radius=1.65pt];
      \ifnum\i=3
        \node[above=5pt,xshift=8pt] at (p\i\j) {$p^{\i,\j}$};
      \else
        \node[above=5pt] at (p\i\j) {$p^{\i,\j}$};
      \fi
    }
  }
  \fill (A) circle[radius=1.65pt];
  \fill (B) circle[radius=1.65pt];
  \node[left=10pt] at (A) {$a$};
   \node[right=10pt] at (B) {$z$};
\end{tikzpicture}
     \caption{A sketch of the graph $D^{3,4}_n$ is shown with the vertices {$(p^{j,k})_{(j,k)\in \{1,...,b\}\times \{1,...,s-1\}}$}.}
     \label{fig:decom}
 \end{figure}

{Accordingly, for each fixed} $j\in \{1,...,b\}$, define the partition of events $(E^{j,k})_{k\in \{0,...,s-1\}}$ where the index $k$ shows how many $p^{j,k}$'s are in $T^z_n$:
\begin{align*}
    &E^{j,k}\coloneq \left\{p^{j,k}\in T^z_n,\, p^{j,k+1}\notin T^z_n\right\},\, \forall k\in \{1,\,...,\,s-2\}\\
   &E^{j,0}=\left\{p^{j,1}\notin T^z_n\right\} , \, E^{j,s-1}=\left\{p^{j,s-1}\in T^z_n\right\}.
\end{align*}
In words, for fixed $j$, $E^{j,k}$ denotes how many {successive junction points in the $j$-th chain belong to $T^z_n$}.
The next lemma shows that, for all fixed $j\in \{1,\,\ldots\,,b\}$, $(E^{j,k})_{k\in \{0,\,\ldots\,,s-1\}}$ is a uniform partition with respect to $\bbP_{{G_n}}$.
\begin{lemma}\label{lm:p in C}
    For all $j\in \{1,\,\dots,\,b\}$ and  $k \in \{0,\,\dots,\,s-1\}$ it holds that
    \begin{equation*}       \bbP_{{G_n}}\left(E^{j,k}\right)=1/s.
    \end{equation*}
\end{lemma}
\begin{proof}
    Fix $j\in \{1,\,\dots,\,b\}$ and $k\in \{1,\,\dots,\,s-1\}$. {The GA} yields that
    \begin{equation}
        \label{eq:GK_j_k}\bbP_{{G_n}}\left(p^{j,k}\in T^z_n\right)=\mathcal{P}^{p^{j,k}}_{D_n}\left(\tau_z<\tau_a\right).
    \end{equation}
    The event {on the right-hand side} depends only on the edges in the $j$-th chain of $D_n$. In one such chain $p^{j,k}$ is the only vertex at distance $ks^{n-1}$ from $z$. Hence, by looking at the projected one-dimensional random walk, the probability {on the right-hand side} of~\eqref{eq:GK_j_k} can be computed through the gambler's ruin problem~\eqref{eq:gambler}:
    \begin{equation}\label{eq:p in C}
       \bbP_{{G_n}}\left(p^{j,k}\in T^z_n\right)=\mathcal{P}^{ks^{n-1}}_{\{0,\dots,s^n\}}\left(\tau_{0}<\tau_{s^n}\right)=\dfrac{s-k}{s}.
    \end{equation}
    Notice that the above result proves the claim for $k=0$ and $k=s-1$. For the other values of $k$, {the GA} and the fact that all the paths leading from $p^{j,k+1} $ to $z$ pass through $p^{j,k}$ lead to
    \begin{equation*}
        \bbP_{{G_n}}\left(p^{j,k+1}\notin T^z_n|p^{j,k}\in T^z_n\right)=\mathcal{P}^{p^{j,k+1}}_{D_n}\left(\tau_a<\tau_{p^{j,k}}\right),
    \end{equation*}
    which {can again be computed} by means of the projected random walk:
    \begin{equation*}        \bbP_{{G_n}}\left(p^{j,k+1}\notin T_n^z|\,p^{j,k}\in T_n^z\right)=\mathcal{P}^{p^{j,k+1}}_{D_n}\left(\tau_a<\tau_{p^{j,k}}\right)=\dfrac{d_n\left(p^{j,k},\,p^{j,k+1}\right)}{d_n\left(a,\,p^{j,k}\right)}=\dfrac{1}{s-k}.
    \end{equation*}
    Using the above equation and~\eqref{eq:p in C} we get the desired result.
\end{proof}
We will now show how to decompose $\mathrm{C}_n$ on such sets. Recall that $c_n$ has been defined in~\eqref{def:c_n} as the number of edges of an element in $\Omega_{D_n}$.
\begin{prop}\label{prop:deco Dbs}
Let $(\mathcal{U}_{i,k})_{i,k \in \mathbb{N}}$ be an independent double-indexed collection of uniform random variables over $\{0,\,...,\,s-1\}$. Define further
\begin{equation*}    
V_{j,\,k}\coloneq c_k\sum_{i=1}^{b^j}\mathcal{U}_{i,k}, \quad k \in \{0,...,n\}.
\end{equation*}
 Then under $\bbP_{{G_n}}$
    \begin{equation}\label{eq:C_n_decomp}
        \mathrm{C}_n\overset{d}{=}\sum_{i=1}^n V_{i,n-i}.
    \end{equation}
\end{prop}
\begin{proof}
    The key observation is that, for all integers $n\geq 2$, a uniform spanning tree of ${G_n}$ is distributed as a vector composed of $b$ independent uniform spanning trees {on a chain (wired at $\{a,\,z\}$).}

    We proceed by induction over $n$. 
    
  We start from the case $n=1$. As just stated,  $T^z_1$ is distributed as $b$ independent uniform spanning trees of each chain. Each chain is a graph with vertex set $\{a,\,p^{j,s-1},\,\dots,\,p^{j,1},\,z\}$  and edge set $\{\overline{ap^{j,s-1}},\,\overline{p^{j,s-1}p^{j,s-2}},\,\dots,\,\overline{p^{j,2}p^{j,1}},\,\overline{p^{j,1}z}\}$ for a fixed $j\in \{1,\,\dots,\,b\}$. Each chain has $s$ spanning tree configurations. Since each configuration is uniquely determined by how many vertices are connected to $z$, the number of vertices in the connected component of $z$ is uniformly distributed in $\{0,\,\dots,\,s-1\}$ without counting $z$ itself. Hence, using the independence of {the configurations in the branches, we get}
    \begin{equation*}    \mathrm{C}_1\overset{d}{=}\sum_{i=1}^{b}\mathcal{U}_{i,0}. 
\end{equation*}

    Suppose now that the decomposition formula holds for $n-1$ and {denote by} $\left(\mathrm{C}_{n-1}(\ell,\,m)\right)_{\ell,m\in\mathbb{N}}$ {a collection of independent copies} of $\mathrm{C}_{n-1}$. {As in the base case, we write} $\mathrm{C}_n$ as a sum of $b$ {independent connected components,} each living in a different chain. For all $j\in \{1,\,\ldots\,,b\}$, decompose the connected component in the $j$-th chain on the partition $\left(E^{j,k}\right)_{k\in \{0,...,s-1\}}$ to get   \begin{equation}\label{eq:prior}
        \mathrm{C}_n\overset{d}{=}\sum_{j=1}^b\sum^{s-1}_{k=0}\ONE\{E^{j,k}\}\left(kc_{n-1}+\mathrm{C}_{n-1}(j,k)\right)\overset{d}{=}\sum_{j=1}^b\left[\mathrm{C}_{n-1}(j,s+j)+\sum_{k=0}^{s-1}\ONE\{E^{j,k}\}kc_{n-1}\right].
    \end{equation}
    Note that the use of $\mathrm{C}_{n-1}(j,s+j)$ {on the right-hand side} is only to indicate that we are summing over fresh independent copies of $\mathrm{C}_{n-1}$.
    {By the inductive hypothesis, $\mathrm{C}_{n-1}$} is distributed as
    \begin{equation*}
        \mathrm{C}_{n-1}\overset{d}{=}\sum_{j=1}^{n-1}c_{n-(j+1)}\sum_{i=1}^{b^j}\mathcal{U}_{i,n-(j+1)}.
    \end{equation*}
    Hence the sum of $b$ independent copies of $\mathrm{C}_{n-1}$ is distributed as
    \begin{equation}\label{eq:u2}    \sum_{j=1}^b\mathrm{C}_{n-1}(j,\,s+j)\overset{d}{=}\sum_{j=1}^{n-1}c_{n-(j+1)}\sum_{i=1}^{b^{j+1}}\mathcal{U}_{i,n-(j+1)}=\sum_{j=2}^nc_{n-j}\sum_{i=1}^{b^j}\mathcal{U}_{i,n-j}.
    \end{equation}
     {Furthermore, by Lemma~\ref{lm:p in C}, $\left\{E^{j,k}\right\}_{k\in \{0,\dots,s-1\}}$} is a uniform partition of the state space, for each $j$. It follows that    \begin{equation}\label{eq:u1}
        \sum_{k=0}^{s-1}\ONE\{E^{j,k}\}kc_{n-1}\overset{d}{=}c_{n-1}\mathcal{U}_{j,n-1}.
    \end{equation}
    Using the expressions obtained in~\eqref{eq:u2} and~\eqref{eq:u1} inside~\eqref{eq:prior}, one obtains that
    \begin{equation*}
     \mathrm{C}_n\overset{d}{=}\sum_{j=2}^nc_{n-j}\sum_{i=1}^{b^j}\mathcal{U}_{i,n-j}+c_{n-1}\sum_{i=1}^b\mathcal{U}_{i,n-1}=\sum_{j=1}^nc_{n-j}\sum_{i=1}^{b^j}\mathcal{U}_{i,n-j},
    \end{equation*}    
    which is exactly the expression in the claim.
\end{proof}
\begin{remark}\label{rem:asymp}
    {Using the decomposition of Proposition~\ref{prop:deco Dbs}, one can} obtain an exact formula for all the moments of $\mathrm{C}_n$. The expected value $\mu_n:=\mathbb{E}\left[\mathrm{C}_n\right]$ can easily be guessed through a symmetry argument: the total number of vertices without counting $a$ and $z$ should be equally split between $T^a_n$ and $T^z_n$, hence $\mu_n=\left(\abs{V_n}-2\right)/2\asymp (bs)^n$. We show how to compute the variance of $\mathrm{C}_n$, which is useful to analyse its asymptotics:
 \begin{equation*}
     \Var\left(\mathrm{C}_n\right)=\sum_{i=1}^nc^2_{n-i}b^i\dfrac{s^2-1}{12}\asymp (bs)^{2n}.
 \end{equation*}

\end{remark}

The decomposition of Proposition~\ref{prop:deco Dbs} {provides a way to analyse} the asymptotics of the distribution of $\mathrm{C}_n$. It is natural to consider its rescaled version:
\begin{equation*}
     \quad Q_n\coloneq \mathrm{C}_n/\mu_n,
\end{equation*}
which takes values in a discrete subset of $\left[0,2\right]$. 

{To emphasise the truly random nature of $\mathrm{Q}_n$, we can} immediately notice that it does not converge in probability to $1$ ({and consequently does not converge to any constant} in $[0,2]$). In fact, since $Q_n$ takes values in a compact set, its convergence to $1$ in probability would imply the convergence to zero of its variance. {However, $\Var\left(\mathrm{Q}_n\right)=\Var\left(\mathrm{C}_n\right)/\mu_n^2$} and in Remark~\ref{rem:asymp} we have shown that $\Var\left(\mathrm{C}_n\right)\asymp\mu_n^2$: hence $\Var\left(\mathrm{Q}_n\right)$ does not vanish in the limit, proving that $\mathrm{Q}_n$ does not converge in probability to a constant.

We proceed by looking for a random limit for $Q_n$. The decomposition through uniform random variables {proved in Proposition}~\ref{prop:deco Dbs} {allows us to compute explicitly} the {moment-generating function} of $\mathrm{Q}_n$, whose limit we determine.

\begin{proof}[Proof of Theorem~\ref{thm:limit Dbs}]
Fix $t\in \mathbb{R}$ and {set $x_{n,i}\coloneq c_{n-i}/\mu_n$}. Using Proposition~\ref{prop:deco Dbs}, the {moment-generating function} of $\mathrm{Q}_n$ {can be written as}
    \begin{equation}\label{eq:log_M}
        \bbE\left[\exp\left(\dfrac{t}{\mu_n}\sum_{i=1}^n V_{i,n-i}\right)\right]=\prod_{i=1}^{n}\left(\bbE\left[\exp\left(tx_{n,i}\,\mathcal{U}\right)\right]\right)^{b^i}=\prod_{i=1}^{n}\left(\dfrac{1}{s}\sum_{k=0}^{s-1}\exp\left(tx_{n,i}k\right)\right)^{b^i}\coloneq M_n(t).
    \end{equation}
    To show the convergence in $n$ of $M_n$, we take logarithms, {which gives}
     \begin{equation*}        \log\left(M_n(t)\right)=\sum_{i=1}^{n}b^i\log\left(\dfrac{1}{s}\sum_{k=0}^{s-1}\exp\left(tx_{n,i}\,k\right)\right).
    \end{equation*}
    For all $i\in \mathbb{N}$, the ratio   $x_{n,i}$ converges to $ {2(bs)^{-i}}$ as $n\to \infty$. We want to apply the dominated convergence theorem. Using the monotonicity of the logarithm, we obtain that
    \begin{equation}\label{eq:dom_conv}
        b^i\log\left(\dfrac{1}{s}\sum_{k=0}^{s-1}\exp\left(tx_{n,i}k\right)\right)\leq b^i|t|x_{n,i}(s-1).
    \end{equation}    
    The ratio $x_{n,i}$ can be explicitly written as
\begin{equation*}
    x_{n,i}=(bs)^{-i}\,2\dfrac{(bs)^nb(s-1)+(bs)^i(b-1)}{\left((bs)^n-1\right)b(s-1)}\leq (bs)^{-i}\,2\dfrac{(bs)^n}{(bs)^n-1}\,\dfrac{bs-1}{b(s-1)}< \,(bs)^{-i}\dfrac{4(bs-1)}{b(s-1)}.
\end{equation*}
Using the above inequality in~{\eqref{eq:dom_conv}, we get} that
\begin{equation*}
    b^i\log\left(\dfrac{1}{s}\sum_{k=0}^{s-1}\exp\left(tx_{n,i}k\right)\right)\leq s^{-i}\,|t|\,\dfrac{4(bs-1)}{b},
\end{equation*}
and the infinite sum (in $i$) of the right-hand side converges for all fixed $t$. This proves {the hypotheses of the dominated convergence theorem, which yields}
\begin{equation*}
    \lim_{n\to \infty}\log\left(M_n(t)\right)=\sum_{i=1}^{\infty}b^i\log\left(\dfrac{1}{s}\sum_{k=0}^{s-1}\exp\left(\dfrac{2tk}{(bs)^i}\right)\right)<\infty.
\end{equation*}
The finiteness of the limit follows from the dominated convergence theorem.
\end{proof}
\begin{remark}\label{rem:infinite_series_Q}
    The limit of $\mathrm{Q}_n$ admits an expression through a random series. Recall that
    \begin{equation*}
        \mathrm{Q}_n\overset{d}{=}\sum_{i=1}^nx_{n,i}\sum_{j=1}^{b^{i}}\mathcal{U}_{i,j}.
    \end{equation*}
    Since for all fixed $i\in \mathbb{N},\, x_{n,i}\to 2(bs)^{-i}$ {as $n\to \infty$, the natural guess} for $\mathrm{Q}$ would be
    \begin{equation*}
        \mathrm{Q}\overset{d}{=}2\sum_{i=1}^{\infty}\dfrac{1}{(bs)^i}\sum_{j=1}^{b^i}\mathcal{U}_{i,j},
    \end{equation*}
    which is {well defined since it converges} absolutely almost surely. It holds that $\mathbb{E}\left[e^{t\mathrm{Q}}\right]=e^{f(t)}$.
    Furthermore, this representation tells us that $\mathbb{P}(\mathrm{Q}=0)=0$, and therefore, for any sequence $(t_n)_{n\in \mathbb{N}}$ with $t_n=o_n(\mu_n)$,
    \[
    \lim_{n\to\infty}\mathbb{P}(\mathrm{C}_n>t_n)=\lim_{n\to\infty}\mathbb{P}\left(\mathrm{Q}_n>\frac{t_n}{\mu_n}\right)=1.
    \]
\end{remark}

\bibliographystyle{abbrvnat}
\bibliography{ref}
\end{document}